\documentclass{amsart}
\usepackage{amssymb}
\usepackage{accents}
\usepackage[utf8]{inputenc}
\usepackage{tikz-cd}
\usepackage{comment}

\numberwithin{equation}{section}
\numberwithin{figure}{section}
\numberwithin{table}{section}
\newtheorem{thm}{Theorem}[section]
\newtheorem*{thm*}{Theorem}
\newtheorem{lem}[thm]{Lemma}
\newtheorem{cor}[thm]{Corollary}
\newtheorem{pro}[thm]{Proposition}

\theoremstyle{definition}

\newtheorem{rem}[thm]{Remark}
\newtheorem{exa}[thm]{Example}

\newcommand{\AG}[1]{\mathbb{A}_{#1}}
\newcommand{\Ass}[1]
{\mathrm{As}(#1)}

\newcommand{\cO}{\mathcal{O}}
\newcommand{\cOd}{\dot{\mathcal{O}}}

\newcommand{\Gal}{\mathrm{Gal}}
\newcommand{\id}{\mathrm{id}}

\newcommand{\lip}[1]{\mathrm{lip}(1)}
\newcommand{\lR}{<_{\cR}}

\newcommand{\ndF}{\mathbb{F}}
\newcommand{\ndN}{\mathbb{N}}
\newcommand{\ndZ}{\mathbb{Z}}
\newcommand{\ord}{\mathrm{ord}}

\newcommand{\PL}{\mathbb{P}}
\newcommand{\PSL}[1]{\mathrm{PSL}{(#1)}}
\newcommand{\PGL}[1]{\mathrm{PGL}{(#1)}}
\newcommand{\ra}{\triangleright}

\newcommand{\SG}[1]{\mathbb{S}_{#1}}

\newcommand{\SL}[1]{\mathrm{SL}{(#1)}}

\newcommand{\cR}{\mathcal{R}}
\newcommand{\cRgen}{\mathcal{R}_0}

\title[Conjugacy classes of $\PSL{2,q}$ as racks]{Subracks and second homology of the conjugacy classes of finite projective special linear groups of degree two}

\author{Istvan Heckenberger}
\address{
Philipps-Universität Marburg, Fachbereich 12 - Mathematik und Informatik,
Hans-Meerwein-Str. 6,
D-35032 Marburg / Germany}
\email{heckenberger@mathematik.uni-marburg.de}

\author{Fengchang Li}
\address{Department of Mathematics, Nanjing University, Nanjing 210093, China}
\email{dg1921005@smail.nju.edu.cn}

\begin{document}

\begin{abstract}
% We determine the subracks, the associated group, and the second homology of the non-trivial conjugacy classes of $\PSL{2,q}$.
We describe the subracks of the conjugacy classes of $\PSL{2,q}$ based on Dickson's theorem on subgroups of $\PSL{2,q}$. All minimal non-abelian subracks of $\PSL{2,q}$ are determined.
Further, we provide a general result on the relationship of associated groups of conjugacy classes in perfect groups to the Schur multiplier of the group. This allows us to conclude explicit descriptions of the associated groups and the second homology of the conjugacy classes of $\PSL{2,q}$ for $q>3$.
\end{abstract}

\maketitle

\section{Introduction}

Racks, their classification and structure theory, including their (co)homology theories, are in the focus of interest for many years, see e.g.~\cite{zbMATH03744146,zbMATH01878448,zbMATH01985908,zbMATH06266759,zbMATH06292119,zbMATH06780322,zbMATH07179896,zbMATH08014168,zbMATH08065981}. It is folklore that rack theory parallels group theory, and that there are very deep connections between the two.

On the other hand, rack theory is a fundamental tool in the study of Hopf and Nichols algebras.
The origins of the present paper are located in the theory of Nichols algebras, see e.g. \cite{zbMATH07269807}, where it is desirable to give structural information on braided Hopf algebras attached to finite groups and classify certain classes of such braided Hopf algebras.
A particularly difficult setting is when the finite group is quasisimple non-abelian and the relevant infinitesimal braiding is a simple Yetter--Drinfeld module, see e.g.~\cite{zbMATH05869018,zbMATH06480688,zbMATH07796386}.
In such settings one often needs (among others) a good knowledge on the subracks of the conjugacy classes of the group and the homology (or better the associated group) of the non-trivial conjugacy classes.
The present paper is \emph{not} dealing with Nichols algebras.
Instead, we are going to collect detailed information on the subracks and associated groups of the conjugacy classes of the finite simple groups $\PSL{2,q}$ with $q\ge 4$ a prime power.
According to \cite[II.7.5]{zbMATH03344733}, an extraordinarily deep work of Thompson tells that all non-abelian simple groups having only solvable proper subgroups are of the form $\PSL{2,q}$ or $\PSL{3,3}$ or are Suzuki groups. (The claim in the reference is even more detailed.) Since Nichols algebras over solvable groups admit a reasonable treatment using deformation techniques, see \cite{arXiv:2411.02304}, the groups $\PSL{2,q}$ are particularly interesting objects from the perspective of Nichols algebras.

We follow two aims. First, we give a detailed description of the subracks of the conjugacy classes in the spirit of Dickson's theorem on subgroups of $\PSL{2,q}$. Second,
with Theorem~\ref{thm:relSchurmult} we formulate a general claim on the relationship between the associated group of a generating conjugacy class of a perfect group on the one side, and the Schur multiplier of the group on the other.
Our claim is an extension of Eisermann's observation in \cite[Th.~1.25]{zbMATH06292119}
for perfect groups with trivial second homology, and seems to be new. Some related (but different) results based on similar techniques can be found e.g.~in
\cite{zbMATH07179896} and \cite{zbMATH08014168}.
We use well-known results on the Schur multiplier of $\PSL{2,q}$ and our description to determine explicitly the associated group of each nontrivial conjugacy class of $\PSL{2,q}$ for all $q>3$. 

We see our presentation as yet another indicator towards further explorations of interactions of rack and group theory.

Following an analogous terminology from group theory, we say that a rack $(X,\ra)$ is
\begin{itemize}
    \item \emph{abelian} if $x\ra y=y$ for all $x,y\in X$, and
    \item \emph{minimal non-abelian} if it is not abelian, but all proper subracks are abelian.
\end{itemize}
Accordingly, we call a conjugacy class of a group \emph{minimal non-abelian} if it is minimal non-abelian as a rack. 
We are particularly interested in the structure of $\Ass X$ and the Nichols algebras over them for minimal non-abelian conjugacy classes of non-abelian (quasi)simple groups. In this paper we identify the minimal non-abelian classes of the simple groups $\PSL{2,q}$.

The main results of the paper are the following:
\begin{itemize}
    \item Propositions \ref{pro:unip_subracks}, \ref{pro:invol_subracks}, \ref{pro:o3_subracks} and \ref{pro:semi_obe4} describe the subracks of different types of conjugacy classes of $\PSL{2,q}$. They are followed by corollaries on minimal non-abelian racks. The proofs are based on Dickson's Theorem (Theorem~\ref{th:Dickson}) and a detailed analysis in Section~\ref{sec:inj_subracks} of injective subracks associated to the involved subgroups.
    \item Theorem~\ref{thm:relSchurmult} describes the associated group of a generating conjugacy class of a perfect group $G$ as a quotient of $\widehat{G}\times \ndZ$, where $\widehat{G}$ is the Schur covering group of $G$.
    \item Theorems~\ref{thm:relSchurPSL} (for $q\ne 9$) and \ref{thm:relSchurPSL29} (for $q=9$) describe concretely for each non-trivial conjugacy class of $\PSL{2,q}$ with $q>4$ the associated group.
\end{itemize}

For any ring $R$ we write $R^\times$ for the group of units of $(R,\cdot)$. The centralizer of an element $g$ of a group $G$ will be denoted by $C_G(g)$.

The first named author would like to thank Victoria Lebed and Leandro Vendramin for valuable discussions.

\section{The category of injective racks}
\label{sec:catracks}

Regarding notations and conventions on racks we follow \cite{zbMATH01985908}, but also other sources mentioned in the introduction.

Recall that a \emph{rack} $(X,\ra)$ is a non-empty set
$X$ together with a binary operation $\ra$ on $X$ such that
\begin{enumerate}
    \item $\varphi_x:X\to X$, $y\mapsto x\ra y$, is bijective for all $x\in X$, and
    \item $x\ra (y\ra z)=(x\ra y)\ra (x\ra z)$
    for all $x,y,z\in X$. (This is called the \emph{self-distributivity of} $\ra $.)
\end{enumerate}
A \emph{rack homomorphism} $f:(X,\ra)\to (Y,\ra)$ is a map $f:X\to Y$ such that
\[ f(x\ra y)=f(x)\ra f(y) \]
for all $x,y\in X$.
Each group $G$ is a rack with $g\ra h=ghg^{-1}$ for all $g,h\in G$. 
A rack $X$ is called \emph{injective} if there exists a group $G$ and a rack monomorphism $X\to G$.

A \emph{subrack} of a rack $(X,\triangleright)$ is a non-empty subset $Y$ of $X$ such that $x\triangleright y\in Y$ for all $x,y\in Y$.

For our purpose it will partially be helpful to work with a category in which we record not only racks but also their embeddings into a group.

Let $\cR$ be the category of pairs $(G,X)$, where $G$ is a group and $X$ is a union of conjugacy classes of $G$.
A morphism between two objects $(H,Y)$ and $(G,X)$ is a group homomorphism $f:H\to G$ such that $f(Y)\subseteq X$. Note that if $Y\ne \emptyset$ then $f\mid_Y:Y\to X$ is a rack homomorphism in this setting.
We call $\cR$ the \emph{category of injective racks}. 
We write
\[ (H,Y) \lR (G,X) \]
for $(H,Y),(G,X)\in \cR$ if there is a morphism $f:(H,Y)\to (G,X)$ which is injective as a group homomorphism. We write $(H,Y)\cong (G,X)$ if there is an isomorphism between $(H,Y)$ and $(G,X)$ in $\cR$. Note that $\lR$ is a transitive relation on the objects of $\cR$.

We are mostly interested in the full subcategory $\cRgen$ of $\cR$, where the objects are the pairs $(G,X)$ such that $X$ generates the group $G$.
Note that each injective rack $X$ can be realized as an object in $\cRgen$. Indeed, if $X$ can be embedded into a group $G$, then $(G(X), X)\in \cRgen$, where $G(X)$ is the subgroup of $G$ generated by $X$.
 There are however nonisomorphic embeddings of injective racks into groups. For example, the rack with one element can be embedded into any cyclic group $C_m$, but $(C_m,\{1\})$ and $(C_n,\{1\})$ are non-isomorphic objects in $\cRgen$ if $m$ and $n$ are different non-negative integers.

For each $m\in \ndZ$ and each $(G,X)\in \cR$, the subset
\[ X^{[m]}=\{ x^m\mid x\in X \}\subseteq G \]
is conjugation invariant. Moreover, any group homomorphism $f:H\to G$ which gives rise to a morphism $f:(H,Y)\to (G,X)$ in $\cR$, gives also rise to a morphism $(H,Y^{[m]})\to (G,X^{[m]})$. This defines an endofunctor $\pi_m:\cR\to \cR$ for each $m\in \ndZ$: $\pi_m(G,X)=(G,X^{[m]})$ and $\pi_m(f)=f$ for all morphisms $f$. Note that
\[ \pi_m\pi_n=\pi_{mn} \] for all $m,n\in \ndZ$. However, $\pi_m$ does not restrict to an endofunctor of $\cRgen$.

For any conjugacy class $X$ in a group $G$ let
\[o(X)\]
denote the order of the elements of the class. If $o(X)$ is finite, then $X^{[m]}$ depends only on the coset of $m$ in $\ndZ/o(X)\ndZ$.
%We will write $[m]$ for $m+o(X)$ in this context.

%The following Lemma is elementary.

%\begin{lem}
%  Let $X$ be a conjugacy class of finite order elements of a group $G$.
%  \begin{enumerate}
%      \item $|X^{[m]}|=|X|$ for all $[m]\in (\ndZ/o(X)\ndZ)^\times $.
%      \item The set
%  \[ a(X)=\{ [m]\in (\ndZ/o(X)\ndZ)^\times \mid X^{[m]}=X\} \]
%  is a subgroup of $\big((\ndZ/o(X)\ndZ)^\times ,\cdot \big)$.
%  \item 
%  Let $p(X)=(\ndZ/o(X)\ndZ)^\times /a(X)$ be the factor group. Then for all cosets $[m]a(X)\in p(X)$ the conjugacy class
%  $X^{[m]}$ of $G$ is independent of the coset representative $[m]$. Moreover, $X^{[m]}=X^{[n]}$ for $m,n\in (\ndZ/o(X)\ndZ)^\times $ if and only if $[m]a(X)=[n]a(X)$ in $p(X)$.
%  \end{enumerate}
%\end{lem}

\section{Conjugacy classes of $\PSL{2,q}$}

Let $q$ be a power of a prime number $p$, let $e=\gcd(2,q-1)$, and let $\ndF_q$ be the finite field with $q$ elements. We write $\ndF_q^{\times 2}$ for the subgroup of squares of the cyclic group $\ndF_q^\times$. This group has order $(q-1)/e$.
In this section we recall some known facts on conjugacy classes of the group $\PSL{2,q}$. For more details and methods we refer to \cite[Ch. II, \S 8]{zbMATH03344733},
\cite[\S5]{MR547113} and \cite[\S1.3]{MR2732651}. 

The groups $\PGL{2,q}$ and $\PSL{2,q}$ consist of the linear fractional transformations
\[ \begin{bmatrix} a & b\\ c & d\end{bmatrix}:\quad  z\mapsto \frac{az+b}{cz+d}
\]
of the projective line $\PL(1,q)$, where $a,b,c,d\in \ndF_q$ with $ad-bc\in \ndF_q^\times $ for elements in $\PGL{2,q}$ and $ad-bc\in \ndF_q^{\times 2}$ for elements in $\PSL{2,q}$. Regarding $\PSL{2,q}$, in each case it is possible to choose such $a,b,c,d$ satisfying the equation $ad-bc=1$. Moreover,
\[ \begin{bmatrix} a & b\\ c & d\end{bmatrix}=
 \begin{bmatrix} a' & b'\\ c' & d'\end{bmatrix}\quad
 \Leftrightarrow \quad \exists
 \lambda \in \ndF_q^\times:
 (a',b',c',d')=(\lambda a,\lambda b,\lambda c,\lambda d).
\]
Correspondingly, each element of $\PGL{2,q}$ has a  characteristic polynomial, which is
\[ [\chi_g(x)]=x^2-(a+d)x+(ad-bc)\in \ndF_q[x]/\sim \quad
\text{ for }g=\begin{bmatrix}
    a & b\\ c & d\end{bmatrix},
\]
where $\sim$ is the equivalence relation on $\ndF_q[x]$ defined by
\[ r(\lambda x)\sim \lambda^n r(x)\quad
\text{for all $\lambda \in \ndF_q^\times $, $r\in \ndF_q[x]$, $\deg\,r=n$.}
\]
Note that the zeros of a characteristic polynomial are only well-defined up to a common non-zero scalar factor; in particular, the number of zeros is well-defined.

The groups $\PGL{2,q}$ and $\PSL{2,q}$ act double transitively on $\PL(1,q)$, and have orders
\[ |\PGL{2,q}|=(q-1)q(q+1),\quad |\PSL{2,q}|=\frac{(q-1)q(q+1)}e. \]
The groups $\PSL{2,q}$ are simple for $q>3$. (For $q\in \{2,3\}$ the group $\PSL{2,q}$ is solvable.) The conjugacy classes of $\PSL{2,q}$ are the following.
\begin{enumerate}
    \item The class of the identity,
    \item the classes of elements with two fixed points in $\PL(1,q)$: $\cO_{2,a}= \begin{bmatrix}
        a & 0\\
        0 & 1
    \end{bmatrix}^{\PSL{2,q}}$, $a\in \ndF_q^{\times 2}\setminus \{1\}$; 
    $\cO_{2,a}=\cO_{2,a'}$ $\Leftrightarrow $ $a'\in \{a,a^{-1}\}$;
    \item the classes $\cO_{1,b}= \begin{bmatrix}
        1 & b\\
        0 & 1
    \end{bmatrix}^{\PSL{2,q}}$, with $b\in \ndF_q^\times$; their elements have one fixed point in $\PL(1,q)$;
    $\cO_{1,b}=\cO_{1,b'}$ $\Leftrightarrow$ $b'\in b\ndF_q^{\times 2}$;
    \item
    the classes $\cO_{0,t}= \begin{bmatrix}
        0 & 1\\
        -1 & t
    \end{bmatrix}^{\PSL{2,q}}$, 
    $t\in \ndF_q\setminus \{x+x^{-1}\mid x\in \ndF_q^\times\}$; their elements have no fixed points; $\cO_{0,y}=\cO_{0,y'}$ $\Leftrightarrow$ $y'\in \{y,-y\}$.
\end{enumerate}

The \emph{split semisimple} classes in (2) have size $q(q+1)$, except $\cO_{2,-1}$ (for odd $q$ with $-1\in \ndF_q^{\times 2}$) which has size $\frac {q(q+1)}2$. The orders of their elements divide $\frac{q-1}e$. Each element of the class $\cO=\cO_{2,a}$ has characteristic polynomial $[\chi_\cO]=x^2-(a+1)x+a$.
A split semisimple class is uniquely determined by its characteristic polynomial.

The \emph{unipotent} classes in (3) have size $\frac{(q-1)(q+1)}e$. The elements in these classes have order $p$. The elements of the unipotent classes have characteristic polynomial $x^2-2x+1$.

The \emph{non-split semisimple} classes in (4) have size $q(q-1)$, except $\cO_{0,0}$ (for odd $q$ with $-1\notin \ndF_q^{\times 2}$), which has size $\frac{q(q-1)}2$.
Occasionally it is helpful to consider these elements inside of $\PSL{2,q^2}$, where they are split semisimple.
The orders of the elements in the non-split semisimple classes are known to divide $\frac{q+1}e$. The characteristic polynomial of the elements of the class $\cO=\cO_{0,t}$ is $[\chi_\cO]=x^2-tx+1$ and has no zeros in $\ndF_q$.

For later reference we collect here some facts and conclusions.

\begin{rem} \label{rem:types_and_orders}
(1)
  The elements of the unipotent classes of $\PSL{2,q}$ have order $p$.
  The orders of the elements of the split semisimple classes are dividing $\frac{q-1}e$
  and are at least two. The orders of the elements of the non-split semisimple classes are dividing $\frac{q+1}e$ and are at least two. The numbers $\frac{q-1}e$ and $\frac{q+1}e$ are relatively prime. Therefore any two classes consisting of elements of the same order have the same semisimplicity type.

  (2) For each $q$, there are, besides 1, up to four different conjugacy class sizes.
  Any two conjugacy classes of the same size have the same semisimplicity type. However, the elements in these two classes may have different orders.

  (3) Knowing the characteristic polynomial of a non-trivial element of $\PSL{2,q}$ suffices to decide which of the three types and which order the element has. The characteristic polynomial determines the conjugacy class uniquely except if $q$ is odd and the class is unipotent.
\end{rem}

We make the above information more explicit for $q=2$ and $q=3$ in Table~\ref{tab:ccsmallq}.

\begin{table}
\begin{tabular}{ccccc}
$q$ & conj.\,class & size &
order of elements & generates $\PSL{2,q}$ \\
\hline
2 & $\cO_{1,1}$ & 3 & 2 & yes \\
& $\cO_{0,1}$ & 2 & 3 & no \\
\hline
3 & $\cO_{1,1}$ & 4 & 3 & yes \\
& $\cO_{1,-1}$ & 4 & 3 & yes \\
& $\cO_{0,0}$ & 3 & 2 & no \\
\hline
\end{tabular}
\label{tab:ccsmallq}
\caption{Non-trivial conjugacy classes of $\PSL{2,2}$ and $\PSL{2,3}$}
\end{table}

Next we give the number of conjugacy classes of elements of a fixed order.
In the claim we use Euler's totient function $\varphi$.

\begin{lem}\ \label{lem:nr_classesbyorder}
  \begin{enumerate}
      \item There is exactly one conjugacy class of involutions in $\PSL{2,q}$.
      \item There are exactly $e$ classes of (unipotent) elements of order $p$.
      \item
 Let $m\in \ndN$ such that $m\ge 3$, and either $m\mid \frac{q-1}e$ or $m\mid \frac{q+1}e$.
 The group $\PSL{2,q}$ has
 $\frac{\varphi(m)}2$ (semisimple) conjugacy classes of elements of order $m$.
  \end{enumerate}
\end{lem}

\begin{proof}
(1) If $p=2$, then involutions are unipotent and there is exactly one such class. Otherwise involutions form a semisimple class: either $\cO_{2,-1}$ (if $4\mid q-1$) or $\cO_{0,0}$ (if $4\mid q+1$).

(2) See the above discussion on unipotent classes.

(3) Assume first that $m\mid \frac{q-1}e$. The elements of $\PSL{2,q}$ of order $m$ are split semisimple.
The split semisimple classes $\cO$ with $o(\cO)=m$ correspond to the sets $\{a,a^{-1}\}$, where $a\in \ndF_q^{\times 2}$ such that $\ord\,a=m$, via the mapping $\cO_{2,a}\mapsto \{a,a^{-1}\}$. Thus the claim follows from the fact that
\[ \# \{a\in \ndF_q^{\times 2}\mid \ord\,a=m\}=\varphi(m)
\]
and that $a\ne a^{-1}$ for all $a\in \ndF_q^{\times 2}$ with $\ord\,a\ge 3$.

Assume now that $m\mid \frac{q+1}e$. Then elements of order $m$ are non-split semisimple, and the classes of such elements correspond to sets $\{\lambda +\lambda^{-1}\}$, where $\lambda \in \ndF_{q^2}^\times $ such that $\ord\,\lambda^2=m$. (Note that then $\lambda^{-1}=\lambda^q$ since $em\mid q+1$. Thus $\lambda +\lambda^{-1}\in \ndF_q$.) Since $\lambda^2\in \ndF_{q^2}^{\times 2}$, and $\lambda +\lambda^{-1}=\mu+\mu^{-1}$ if and only if $(\lambda -\mu)(\lambda \mu-1)=0$, the claim follows as for split semisimple classes.
\end{proof}

Next we consider racks with special properties. Recall that a rack $X$ is called abelian if $x\ra y=y$ for all $x,y\in X$.

\begin{rem} \label{rem:abelian_classes}
  A conjugacy class of $\PSL{2,q}$ is abelian if and only if $q\in \{2,3\}$ and the class is non-split semisimple. These are the class of 3-cycles in $\SG 3\cong \PSL{2,2}$ and the class of involutions in $\AG 4\cong \PSL{2,3}$. 
\end{rem}

Recall that a conjugacy class $\cO$ of a finite group is \emph{real} if $g^{-1}\in \cO$ for all $g\in \cO$. The next claim is well-known and easy to check.

\begin{lem} \label{lem:realclasses}
  The real classes of $\PSL{2,q}$ are
  \begin{enumerate}
      \item the (unipotent) class of involutions if $q$ is even,
      \item
      all unipotent classes if $4\mid q-1$, and
      \item all semisimple classes.
 \end{enumerate}
 Moreover, the centralizer of any element $g$ with $g^2\ne 1$ in a semisimple class $\cO$ is $C_{\cO}(g)=\{g,g^{-1}\}$.  
\end{lem}

Later we will need more detailed information about equality of powers of conjugacy classes.

\begin{pro} \label{pro:equalpowers}
  Let $\cO$ be a semisimple class of $\PSL{2,q}$ and let $m\in \ndZ$ such that $\gcd(m,o(\cO))=1$.
  Then $\cO^{[m]}=\cO$ if and only if $o(\cO)\mid m-1$
  or $o(\cO)\mid m+1$.
\end{pro}

Note that for $m=-1$ the proposition says that $\cO^{[-1]}=\cO$, that is, that semisimple classes are real.

\begin{proof}
  Assume first that $\cO$ is split semisimple. Let $a\in \ndF_q^{\times 2}$ such that
  $\cO=\cO_{2,a}$. Then $\cO^{[m]}=\cO_{2,a^m}$. Hence
  $\cO^{[m]}=\cO$ if and only if $a^m\in \{a,a^{-1}\}$. This implies the claim of the proposition since $o(\cO)=\ord\,a$.

  Assume now that $\cO$ is non-split semisimple. Then $\cO^{[m]}=\cO$ if and only if the characteristic polynomials of $\cO^{[m]}$ and $\cO$ coincide. By embedding $\cO$ into the group $\PSL{2,q^2}$ we conclude from the previous paragraph that the latter holds if and only if $o(\cO)\mid m-1$ or $o(\cO)\mid m+1$.
\end{proof}

The structure of conjugacy classes of $\PGL{2,q}$ is very similar. They are
\begin{enumerate}
    \item the class of the identity,
    \item the classes of elements with two fixed points in $\PL(1,q)$: $\cOd_{2,a}= \begin{bmatrix}
        a & 0\\
        0 & 1
    \end{bmatrix}^{\PGL{2,q}}$, $a\in \ndF_q^\times \setminus \{1\}$; 
    $\cOd_{2,a}=\cOd_{2,a'}$ $\Leftrightarrow $ $a'\in \{a,a^{-1}\}$;
    \item the class $\cOd_1= \begin{bmatrix}
        1 & 1\\
        0 & 1
    \end{bmatrix}^{\PGL{2,q}}$; its elements have one fixed point in $\PL(1,q)$;
    \item
    the classes $\cOd_{0,t,d}= \begin{bmatrix}
        0 & 1\\
        -d & t
    \end{bmatrix}^{\PGL{2,q}}$, where
    $z^2-tz+d\in \ndF_q[z]$ is irreducible; their elements have no fixed points; $\cOd_{0,t,d}=\cOd_{0,t',d'}$ $\Leftrightarrow$ $\exists \lambda \in \ndF_q^\times$: $(t',d')=(\lambda t,\lambda^2 d)$.
\end{enumerate}
Again, $|\cOd_{2,a}|=q(q+1)$ for all $a\in \ndF_q^\times \setminus \{1,-1\}$, and $|\cOd_{2,-1}|=\frac{q(q+1)}2$;
$|\cOd_{1}|=q^2-1$;
$|\cOd_{0,t,d}|=q(q-1)$ if $t\ne 0$, and $|\cOd_{0,0,d}|=\frac{q(q-1)}2$.

\begin{rem} \label{rem:PGLvsPSL}
(1) The characteristic polynomials of split semisimple elements of $\PGL{2,q}$ have two distinct zeros in $\ndF_q$ (which are only well-defined up to a common scalar factor). Unipotent elements have characteristic polynomial $[\chi]=x^2-2x+1$. Non-split semisimple elements have an irreducible characteristic polynomial. All elements of a conjugacy class of $\PGL{2,q}$ have the same characteristic polynomial. The characteristic polynomial of a semisimple class determines the conjugacy class uniquely.

(2) The classes of $\PGL{2,q}$ either intersect trivially with $\PSL{2,q}$ or they are unions of classes of $\PSL{2,q}$. More precisely, $\cOd_{2,a}=\cO_{2,a}$ for all $a\in \ndF_q^{\times 2}\setminus \{1\}$; $\cOd_1=\cO_{1,1}$ if $q$ is even, and
$\cOd_1=\cO_{1,1}\sqcup \cO_{1,b}$ for any $b\in \ndF_q^\times \setminus \ndF_q^{\times 2}$ otherwise;
$\cOd_{0,t,1}=\cO_{0,t}$
for all $t\in \ndF_q$.
\end{rem}

\begin{rem} \label{rem:imageQ}
Let $Q(x)=x^2-tx+d\in \ndF_q[x]$ be an irreducible polynomial. Then
\[ \big\{ y^2-tyz+dz^2\mid (y,z)\in \ndF_q^2\setminus \{(0,0)\} \big\}=\ndF_q^\times. \]
Indeed, the left hand side is a union of  $\ndF_q^{\times 2}$-cosets of $\ndF_q^\times$, and for odd $q$ it consists of at least $(q+1)/2$ elements at $z=1$.
\end{rem}

\section{Subracks of the conjugacy classes}

Let $G$ be a group and let $X$ be a union of conjugacy classes of $G$. Then each subrack $Y$ of $X$ is a union of conjugacy classes of the subgroup $G(Y)$ generated by $Y$. Conversely, let $H$ be a subgroup of $G$ and let $\cO$ be a subset of $H\cap X$ which generates $H$. If $\cO$ is non-empty and a union of conjugacy classes of $H$, then $\cO $ is a subrack of $X$. In this section, we use this correspondence to list the subracks of conjugacy classes of the simple groups $\PSL{2,q}$ up to isomorphism.
Our main tool is Dickson's theorem on the subgroups of $\PSL{2,q}$. %Some useful additional facts in this context can be found in \cite{MR2853635}.

Recall that $p$ is a prime, $q$ is a power of $p$, and $e=\gcd(q-1,2)$. We are going to state a slightly improved version of Dickson's theorem in
\cite[Thm.\,8.27]{zbMATH03344733}. There we will use the following groups.

For each subgroup $C$ of $\ndF_q^{\times 2}$ and each vector space $A$ over $\ndF_{q_0}$, where $\ndF_{q_0}$ is the smallest subfield of $\ndF_q$ containing $C$, let
\[ A\rtimes_q C\]
be the semidirect product defined by the action of $C$ on $A$ given by left multiplication. If $|A|\le q$, then $A$ can be chosen as a subgroup of $\ndF_q$, and $A\rtimes _qC$ can be realized as a subgroup of $\PSL{2,q}$:
\[ A\rtimes_q C \cong 
\left\{
 \begin{bmatrix}
  c & a \\ 0 & 1
 \end{bmatrix}\,
 \mid \,
 a\in A,c\in C 
 \right\}\subseteq \PSL{2,q}.
\]
Note that
\[ |C|\,\mid \, \gcd\left(|A|-1,\frac{q-1}e\right).
\]

\begin{thm} (Dickson's theorem)
\label{th:Dickson}
  Each subgroup of $\PSL{2,q}$ belongs to the following list:
  \begin{enumerate}
      \item[(1)] elementary abelian $p$-groups of order dividing $q$,
      \item[(2)] cyclic groups of order dividing $\frac{q+1}e$ or $\frac{q-1}e$,
      \item[(3)] dihedral groups of order dividing $\frac{2(q+1)}e$ or $\frac{2(q-1)}e$,
      \item[(4)] groups isomorphic to the alternating group $\AG 4$ for $q$ odd or a power of 4,
      \item[(5)] groups isomorphic to the symmetric group $\SG 4$ if $16$ divides $q^2-1$,
      \item[(6)] groups isomorphic to the alternating group $\AG 5$ if $p=5$ or $5$ divides $q^2-1$,
      \item[(7)]
      groups isomorphic to
      $A\rtimes_q C$ with $C\ne \{1\}$, $1<|A|\le q$,
      \item[(8)] groups isomorphic to $\PSL{2,q_0}$, where $q$ is a power of $q_0$, or to $\PGL{2,q_0}$, where $q$ is a power of $q_0^2$.      
  \end{enumerate}
  Conversely, each group in the above list is isomorphic to a subgroup of $\PSL{2,q}$.
\end{thm}

\begin{rem} \label{rem:subgroups_up_to_conj}
  (1) The transformations
  $\begin{bmatrix} 1 & b \\ 0 & 1\end{bmatrix}$, $b\in \ndF_q$, form a $p$-Sylow subgroup $U(2,q)$ of $\PSL{2,q}$, and this group is elementary abelian of order $q$. Hence all $p$-subgroups in Theorem~\ref{th:Dickson}(1) are conjugate to a subgroup of $U(2,q)$.
  
  (2) In \cite[Thm.\,8.27(7)]{zbMATH03344733} a class of semidirect product groups is listed, but not all of them appear as subgroups of $\PSL{2,q}$. Note also that in some references to Dickson's theorem there is a mistake in the description of these subgroups. In fact, each subgroup of $\PSL{2,q}$, which is a semidirect product of a non-trivial elementary abelian $p$-group and a non-trivial cyclic group, is conjugate to a group $A\rtimes_q C$ as in Theorem~\ref{th:Dickson}(7). 
    This follows from (1) applied to the elementary abelian $p$-group and by considering its normalizer in $\PSL{2,q}$.
    In particular, all elements of such a group have order $p$ or an integer dividing $|C|$.
\end{rem}

\begin{rem}
\label{rem:PSLgeneration}
(Generation of $\PSL{2,q}$) The group $\PSL{2,q}$ is generated as a group by $s=\begin{bmatrix} 0 & 1\\ -1 & 0 \end{bmatrix}$ and the $p$-Sylow subgroup $U(2,q)$ in Remark~\ref{rem:subgroups_up_to_conj}(1), see
\cite[Lemma~1.2.2]{MR2732651} for the group $\SL{2,q}$ with odd $q$; the proof works also for even $q$.
\end{rem}

The following Proposition will be helpful in our study of intersections of conjugacy classes and subgroups.

\begin{pro} \label{pro:ACintoPSL2}
  Let $f:A\rtimes_q C\to \PSL{2,q}$ be an injective group homomorphism and
  let  $\ndF_{q_0}$ be the subfield of $\ndF_q$ generated by $C$. Then there exist $g\in \PSL{2,q}$, an automorphism $\Phi\in \Gal(\ndF_{q_0}/\ndF_p)$, and a $\Phi$-semilinear map $\Phi_A:A\to \ndF_q$ such that
  \[ g\,f(a,c)g^{-1}=
  \begin{bmatrix}
  \Phi(c) & \Phi_A(a)\\ 0 & 1 
  \end{bmatrix}
  \]
  for all $a\in A$, $c\in C$.
\end{pro}

\begin{proof}
The map $f$ sends $A$ into an elementary abelian $p$-subgroup of $\PSL{2,q}$, and $C$ into its normalizer. By passing to a conjugate morphism, as observed in Remark~\ref{rem:subgroups_up_to_conj}(2), we may assume that there exist maps $\Phi_1:C\to \ndF_q^{\times 2}$ and $\Phi_2:A\to \ndF_q$
such that
\[ f(a,c)=\begin{bmatrix}
    \Phi_1(c) & \Phi_2(a) \\
    0 & 1
\end{bmatrix}\]
for all $a\in A$, $c\in C$. Since $f$ is a group homomorphism, it follows that $\Phi_1$ and $\Phi_2$ are group homomorphisms and that $\Phi_2(ca)=\Phi_1(c)\Phi_2(a)$ for all $c\in C$, $a\in A$. Moreover, $\Phi_1$ and $\Phi_2$ are injective since $f$ is injective.

Let $r(x)=\sum_{i=0}^m c_ix^i\in \ndF_p[x]$ be the minimal polynomial of some $c\in \ndF_q^{\times 2}$. Then
\[ 0=\Phi_2\left(\sum _{i=0}^m c_ic^i a\right)=\sum_{i=0}^m \Phi_1(c)^ic_i\Phi_2(a) \] for each $a\in A$. Since $A\ne 0$ and $\Phi_2$ is injective, it follows that $r(\Phi_1(c))=0$. Hence $\Phi_1$ is the restriction of some $\Phi\in \Gal(\ndF_{q_0}/\ndF_p)$. It also follows then that $\Phi_2$ is a $\Phi$-semilinear map, which proves the claim.
\end{proof}

\begin{pro} \label{pro:PSLintoPSL2}
  Let $f:\PSL{2,q_0}\to \PSL{2,q}$
  be an injective group homomorphism, where
  $q_0\in \ndN$ and $q$ is a power of $q_0$. Then there exist $g\in \PSL{2,q}$, an automorphism $\Phi\in \Gal(\ndF_q/\ndF_p)$, and an $x\in \ndF_q^\times $ such that
  \[ f\left( \begin{bmatrix} a & b \\ c & d \end{bmatrix}\right)=
  g^{-1}\begin{bmatrix}
      \Phi(a) & \Phi(b)x \\
      \Phi(c)x^{-1} & \Phi(d)
      \end{bmatrix}g
  \]
  for all $\begin{bmatrix}
      a & b\\ c & d\end{bmatrix}\in \PSL{2,q_0}$.
\end{pro}

\begin{proof}
  The subgroup $U(2,q_0)$ of translations $x\mapsto x+a$, $a\in \ndF_{q_0}$, of $\PSL{2,q_0}$ is a non-trivial elementary abelian $p$-group. By Remark~\ref{rem:subgroups_up_to_conj}(1) we may pass to a conjugate morphism and assume that $f(U(2,q_0))$ consists of transformations with fixed point $\infty$ in $\PL(1,q)$. We consider the elements
  \[ s=\begin{bmatrix} 0 & 1\\ -1 & 0 \end{bmatrix},\quad
  t=\begin{bmatrix} 1 & 1\\ 0 & 1 \end{bmatrix}
\]
  in $\PSL{2,q_0}$. If $f(s)$ would have $\infty$ as a fixed point, then $f(s)f(t)f(s)$ would commute with $f(t)$.
  However, $sts$ and $t$ do not commute, and $f$ is injective by assumption. Hence $\infty$ is not fixed by the involution $f(s)$, and by double transitivity of the action of $\PSL{2,q}$ on $\PL(1,q)$ we may assume that $f(s)$ permutes $0$ and $\infty$.
  Thus there exists $x\in \ndF_q^\times$
  such that $f(s)=
  \begin{bmatrix} 0 & x\\ -x^{-1} & 0 \end{bmatrix}$.
  Therefore, by the argument in the proof of Proposition~\ref{pro:ACintoPSL2},
  \[ f\left(\begin{bmatrix} a & b \\ c & d \end{bmatrix}\right) =\begin{bmatrix}
      \Phi(a) & \Phi(b)x \\
      \Phi(c)x^{-1} & \Phi(d)
  \end{bmatrix}\]
  for some $\Phi\in \Gal(\ndF_q/\ndF_p)$,
  whenever $c=0$, $d=1$, or $a=d=0$, $b=-c=1$.
  This implies the claim of the Proposition, since $U(2,q_0)$ and $s$ generate $\PSL{2,q_0}$ by Remark~\ref{rem:PSLgeneration}.
\end{proof}

\begin{rem}
\label{rem:PGLintoPSL2}
   Let $f:\PGL{2,q_0}\to \PSL{2,q}$
  be an injective group homomorphism, where
  $q_0\in \ndN$ and $q$ is a power of $q_0^2$. Then the claim of Proposition~\ref{pro:PSLintoPSL2} is still valid (where the last equation holds for all elements of $\PGL{2,q_0}$). The proof of the Proposition is valid also in this case without essential modifications.
\end{rem}

% \begin{rem} \label{rem:subgroup_generation}
%   Subgroups of $\PSL{2,q}$ will be generated in our context by subsets of conjugacy classes. In particular, all generators have the same order $m$. Such generating subsets for the subgroups $H$ in Dickson's theorem exist only in the following situations:
%   \begin{enumerate}
%       \item[(1)] $m=p$
%       \item[(2)] $m=|H|$
%       \item[(3)] $m=2$
%       \item[(4)] $m=3$
%       \item[(5)] $m=2$ or $m=4$
%       \item[(6)] $m=2$ or $m=3$ or $m=5$
%       \item[(7)] $m=|C|$, where $H=A\rtimes C$
%       \item[(8a)] when $H\cong \PSL{2,2}$: $m=2$;
%       $H\cong \PSL{2,3}$:
%       $m=3$;
%       $H\cong \PSL{2,q_0}$,
%       $q_0\ge 4$: $m=p$ or $m$ divides $\frac{q_0-1}e$ or $m$ divides $\frac{q_0+1}e$
%       \item[(8b)] when $H\cong \PGL{2,3}\cong \SG 4$: $m\in \{2,4\}$; 
%       $H\cong \PGL{2,q_0}$, $q_0\ge 5$ odd: $m$ divides $q_0-1$ or $m$ divides $q_0+1$.
%   \end{enumerate}
% \end{rem}

%Before studying the %subracks of conjugacy %classes of $\PSL{2,q}$ in %detail, we discuss rack %isomorphisms between %conjugacy classes.

%The following lemma is %elementary.

% \begin{lem} \label{lem:powerrackiso}
%     Let $G$ be a group and let $g\in G$ be an element of finite order $m$. Then for all $k\in \ndZ$ with $\gcd (m,k)=1$, the map
%     \[ \psi_k: g^G\to (g^k)^G,\quad x\mapsto x^k, \]
%     is a rack isomorphism with inverse $\psi_l$, where $l\in \ndZ $ such that $kl\equiv 1\, \mathrm{mod}\, m$.
% \end{lem}

\section{Injective subracks of conjugacy classes}
\label{sec:inj_subracks}

In this section let $\cO$ be a non-trivial conjugacy class of $\PSL{2,q}$ with $q\ge 2$. We discuss injective subracks of $(\PSL{2,q},\cO)$.
 The following two general lemmas will apply frequently.

\begin{lem} \label{lem:lR_oneclass}
    Let $H$ be a subgroup of a group $G$ and let $m\in \ndN$. Assume that there is precisely one conjugacy class $\cO$ of $G$ such that $o(\cO)=m$. Then $(H,Y)\lR (G,\cO)$ for each non-empty conjugation invariant subset $Y$ of $H$ such that $\ord\,y=m$ for all $y\in Y$. 
\end{lem}

\begin{lem} \label{lem:lR_isoclasses}
    Let $H$ be a subgroup of a group $G$ and let $m\in \ndN$. Assume that up to isomorphism of injective racks there is precisely one conjugacy class $\cO$ of $G$ such that $o(\cO)=m$. Then $(H,Y)\lR (G,\cO)$ for each conjugacy class $Y$ of $H$ such that $o(Y)=m$.
\end{lem}

We now look at injective subracks of $\cO$ corresponding to specific subgroups of $\PSL{2,q}$.

\begin{lem} \label{lem:elemabsubrack}
Let $(H,Y)\in \cRgen$ be such that $H$ is an elementary abelian $p$-group.
The following are equivalent.
\begin{enumerate}
    \item $(H,Y)\lR (\PSL{2,q},\cO)$.
    \item There exist $x\in \ndF_q^\times$ and a  non-empty subset $\tilde{Y}$ of $\ndF_q^{\times 2}$ such that
    $\cO=\cO_{1,x}$ and $Y$ is conjugate to
    \[
    \left \{ \begin{bmatrix} 1 & bx \\ 0 & 1\end{bmatrix}\,
    \Big|\,b\in \tilde{Y} \right\}.
    \]
\end{enumerate}
\end{lem}

\begin{proof}
   The normalizer of $U(2,q)$ is the subgroup of transformations
   $\begin{bmatrix} a & b \\ 0 & d\end{bmatrix}$
   with $ad\in \ndF_q^{\times 2}$. In particular, $\begin{bmatrix} 1 & b \\ 0 & 1\end{bmatrix}$ and
   $\begin{bmatrix} 1 & b' \\ 0 & 1\end{bmatrix}$ for $b,b'\in \ndF_q^\times$ are conjugate if and only if $b'\in b\ndF_q^{\times 2}$.
   Now the claim follows from Remark~\ref{rem:subgroups_up_to_conj}(1).
\end{proof}

\begin{lem} \label{lem:subrackcyclic}
Let $(H,Y)\in \cR$ be such that $H$ is a non-trivial cyclic group of order dividing $\frac{q-1}e$ or $\frac{q+1}e$.
The following are equivalent.
\begin{enumerate}
    \item $(H,Y)\in \cRgen$ and $(H,Y)\lR (\PSL{2,q},\cO)$.
    \item The class $\cO$ is (split or non-split) semisimple, and $Y\subseteq \{y,y^{-1}\}$ for some $y\in \cO$ with $\ord\,y=|H|$.
\end{enumerate}
\end{lem}

Recall from Lemma~\ref{lem:realclasses} that split and non-split semisimple classes are real. Hence, in (2), $y^{-1}\in \cO$ whenever $y\in \cO$.

\begin{proof}
   Clearly, (2) implies (1). Conversely, assume (1). Then the elements of $H$ have order dividing $\frac{q-1}e$ or $\frac{q+1}e$. Hence $\cO$ is semisimple. If $|H|=2$, then clearly (2) holds.
   Assume now that $|H|>2$, and let $y\in Y$. Then $Y\subseteq C_{H\cap \cO}(y)$, since $H$ is abelian. Hence $Y\subseteq \{y,y^{-1}\}$ by the last part of Lemma~\ref{lem:realclasses}. The equation $\ord\,y=|H|$ holds since $Y$ generates $H$.
\end{proof}

We denote by $D_{2n}$ the dihedral group of order $2n$.

\begin{lem}
\label{lem:dihedralsubracks}
Let $(D_{2n},Y)\in \cR$, where $n\ge 2$.
Assume that $n$ divides $\frac{q-1}e$ or $\frac{q+1}e$.
The following are equivalent.
\begin{enumerate}
    \item $(D_{2n},Y)\in\cRgen$ and $(D_{2n},Y)\lR (\PSL{2,q},\cO)$.
    \item The class $\cO$ is the unique class of involutions in $\PSL{2,q}$. Moreover, either
    $Y\subseteq D_{2n}$ is the set of involutions, or $n$ is even and $Y\subseteq D_{2n}$ is the union of two classes of involutions of size $n/2$.
\end{enumerate}
If these conditions are fulfilled then either $|Y|=n$, or $n$ is even and $|Y|=n+1$.
\end{lem}

\begin{proof}
  Assume (1). All elements of $Y$ have the same order and $Y$ generates $D_{2n}$, and hence $Y$ consists of involutions. Therefore $\cO$ is the unique class of involutions of $\PSL{2,q}$. Again, since $Y$ generates $D_{2n}$ and is conjugation invariant, the claims on $Y$ in (2) hold. Indeed, if $n$ is odd, then all involutions of $D_{2n}$ are conjugate. On the other hand, if $n$ is even, then there are two classes of involutions of size $n/2$, and an additional central involution. Note that one class of involutions together with a central involution generates $D_{2n}$ only if $n=2$.
    
  Conversely, assume that (2) holds. Then $D_{2n}$ is generated by $Y$. Moreover, $D_{2n}$ is a subgroup of $\PSL{2,q}$ by Dickson's theorem. Since $\cO$ is the unique class of involutions of $\PSL{2,q}$, Lemma~\ref{lem:lR_oneclass} implies (1).
\end{proof}

We recalled the real classes of $\PSL{2,q}$ in Lemma~\ref{lem:realclasses}.

\begin{lem} \label{lem:A4classesinPSL}
Let $(\AG 4,Y)\in \cR$ and let $G=\PSL{2,q}$. The following are equivalent.
\begin{enumerate}
\item $(\AG 4,Y)\in \cRgen$ and $(\AG 4,Y)\lR (G,\cO)$.
\item 
$o(\cO)=3$, $q$ is odd or a power of 4, and either $(\AG 4, Y)\cong (\AG 4,(1\,2\,3)^{\AG 4})$, or
$\cO$ is real and  $(\AG 4,Y)=(\AG 4,(1\,2\,3)^{\AG 4}\cup (1\,3\,2)^{\AG 4})$.
\end{enumerate}
\end{lem}

\begin{proof}
Suppose first that (1) holds. Then $\ord \,y=o(\cO)$ for all $y\in Y$. Hence $o(\cO)=3$, and either
$(\AG 4, Y)\cong (\AG 4,(1\,2\,3)^{\AG 4})$
or
$(\AG 4, Y)=(\AG 4,(1\,2\,3)^{\AG 4}\cup (1\,3\,2)^{\AG 4})$. Moreover,
$q$ is odd or a power of 4 by
Dickson's theorem.
Since $(1\,3\,2)=(1\,2\,3)^{-1}$, the relation
$(\AG 4,(1\,2\,3)^{\AG 4}\cup (1\,3\,2)^{\AG 4})\lR (G,\cO)$ also implies that $\cO$ is real. This proves (2).

Now assume that (2) holds. Then $\AG 4$ is a subgroup of $G$ by Dickson's theorem.
If $p\ne 3$, then $\cO$ is semisimple, real, and the only class of $G$ with $o(\cO)=3$: $\cO=\cO_{2,a}$ with $a^3=1$ or $\cO=\cO_{0,1}$. Hence
$(\AG 4,Y)\lR (G,\cO)$ by Lemma~\ref{lem:lR_oneclass}.
Finally, if $p=3$, then $\cO$ is unipotent, and
$(G,\cO)\cong (G,\cO')$ for the second unipotent class $\cO'$ of $G$ via
conjugation in $\PGL{2,q}$. Therefore $(\AG 4,(1\,2\,3)^{\AG 4})\lR (G,\cO)$ by Lemma~\ref{lem:lR_isoclasses}. If additionally $\cO$ is real, then any embedding of $\AG 4$ into $G$ sends $(1\,2\,3)^{\AG 4}\cup (1\,3\,2)^{\AG 4}$ into a single class, and again the isomorphy $(G,\cO)\cong (G,\cO')$ for the second unipotent class $\cO'$ implies that  
$(\AG 4,(1\,2\,3)^{\AG 4}\cup(1\,3\,2)^{\AG 4})\lR (G,\cO)$.
\end{proof}

\begin{lem} \label{lem:S4classesinPSL}
Let $(\SG 4,Y)\in \cR$ and let $G=\PSL{2,q}$. The following are equivalent.
\begin{enumerate}
\item $(\SG 4,Y)\in \cRgen$ and $(\SG 4,Y)\lR (G,\cO)$.
\item $16\mid q^2-1$, and one of the following sets of conditions is fulfilled:
\begin{enumerate}
    \item
    $o(\cO)=2$, and $Y$ is one of the subsets $(1\,2)^{\SG 4}$ and $(1\,2)^{\SG 4}\cup (1\,2)(3\,4)^{\SG 4}$.
    \item $o(\cO)=4$ and 
    $Y=(1\,2\,3\,4)^{\SG 4}$.
\end{enumerate}
\end{enumerate}
\end{lem}

\begin{proof}
Suppose first that (1) holds. Then
$16\mid q^2-1$ by
Dickson's theorem.
Moreover,
$\ord \,y=o(\cO)$ for all $y\in Y$. Since $Y$ generates $\SG 4$, we conclude that $o(\cO)\in \{2,4\}$, and $Y$ is as in (2). Conversely, assume  (2). Clearly, $(\SG 4,Y)\in \cRgen$. Moreover, $\SG 4$ embeds into $G$ by Dickson's theorem and since $16\mid q^2-1$. Since $q$ is odd, the classes of $G$ with even order elements are real. In particular, there is exactly one class of involutions and one class of elements of order 4. Hence
$(\SG 4,Y)\lR (G,\cO)$ by Lemma~\ref{lem:lR_oneclass}.
\end{proof}

\begin{lem} \label{lem:A5classesinPSL}
Assume that $p\ne 5$. Let $G=\PSL{2,q}$ and let $(\AG 5,Y)\in \cR$. The following are equivalent.
\begin{enumerate}
\item $(\AG 5,Y)\in \cRgen$ and $(\AG 5,Y)\lR (G,\cO)$.
\item $5\mid q^2-1$, and one of the following sets of conditions is fulfilled:
\begin{enumerate}
    \item
    $o(\cO)=2$ and $Y=(1\,2)(3\,4)^{\AG 5}$.
    \item $o(\cO)=3$ and 
    $Y=(1\,2\,3)^{\AG 5}$. 
    \item $o(\cO)=5$ and $(\AG 5,Y)\cong (\AG 5,(1\,2\,3\,4\,5)^{\AG 5})$.
\end{enumerate}
\end{enumerate}
\end{lem}

Note that $\PSL{2,5}\cong \AG 5$. In Lemma~\ref{lem:PSLq0} we will discuss embeddings of $\AG 5$ into $\PSL{2,q}$, where $q$ is a power of $5$, as embeddings of $\PSL{2,5}$ into $\PSL{2,q}$.

\begin{proof}
Assume (1). Then $5\mid q^2-1$ by Dickson's theorem, since $p\ne 5$.
The non-trivial conjugacy classes of $\AG 5$ are
\[ (1\,2)(3\,4)^{\AG 5}, \quad (1\,2\,3)^{\AG 5},\quad
X_1=(1\,2\,3\,4\,5)^{\AG 5},
\quad X_2=(1\,2\,3\,5\,4)^{\AG 5}. \]
Thus $o(\cO)=o(Y)\in \{2,3,5\}$, and if $o(\cO)\in \{2,3\}$, then (2)(a) or (2)(b) hold.

Assume that $o(\cO)=5$. 
Then $\cO$ is semisimple (since $p\ne 5$), and hence 
$\cO^{[2]}\ne \cO$ by Proposition~\ref{pro:equalpowers}.
Moreover, $X_2=X_1^{[2]}$. Hence $H\cap \cO\not\cong X_1\cup X_2$ for each subgroup $H\cong \AG 5$ of $\PSL{2,q}$. Moreover, $X_2\cong X_1$ via conjugation by a transposition in $\SG 5$. Thus (2)(c) holds.

 Assume now (2). Then $\AG 5$ is a subgroup of $G$ by Dickson's theorem. If $o(\cO)=2$, then $(\AG 5,(1\,2)(3\,4)^{\AG 5})\lR (G,\cO)$ by Lemma~\ref{lem:lR_oneclass}. Similarly, if $o(\cO)=3$, then $(\AG 5,(1\,2\,3)^{\AG 5})\lR (G,\cO)$ by Lemma~\ref{lem:lR_oneclass} (when $p\ne 3$) and by Lemma~\ref{lem:lR_isoclasses} (when $p=3$).
 Finally, in case (2)(c) the class $\cO$ is semisimple since $p\ne 5$.
 In $G$ there are two classes of elements of order 5 by Lemma~\ref{lem:nr_classesbyorder}(3). By the argument in the previous paragraph, each subgroup $H$ of $G$ with $H\cong \AG 5$ intersects both $\cO$ and $\cO^{[2]}$ in a single class of $\AG 5$. Since
 $(\AG 5,Y^{[2]})\cong (\AG 5,Y)$, it follows that
 $(\AG 5,Y)\lR (G,\cO)$. This proves (1).
\end{proof}

In the following two claims, $A\rtimes_qC$ denotes a group as in Theorem~\ref{th:Dickson}(7). The proof of the first lemma is straightforward.

\begin{lem}\
\label{lem:ACclasses}
  \begin{enumerate}
  \item
  Let $Y$ be a conjugacy class of $A\rtimes_q C$. Then either $o(Y)=p$ and $Y\subseteq A$, or $Y=A\rtimes_q \{c\}$
  for some $c\in C\setminus \{1\}$,
  and $o(Y)=\ord\,c$.
  \item
  Let $Y$ be a non-empty union of conjugacy classes of $A\rtimes_q C$ such that $\ord\,x=\ord\,y$ for all $x,y\in Y$. Then $Y$ generates $A\rtimes _qC$ if and only if $\ord\,y=|C|$ for all $y\in Y$.
  \end{enumerate}
\end{lem}

\begin{lem}
\label{lem:ACsubracksinPSL}
Let  $(A\rtimes_q C,Y)\in \cR$. The following are equivalent.
\begin{enumerate}
\item $(A\rtimes_q C,Y)\in \cRgen $ and $(A\rtimes_q C,Y)\lR (\PSL{2,q},\cO)$.
\item $\cO=\cO_{2,z}$ for some $z\in C$ with $\ord\,z=|C|$, and one of the following hold.
\begin{enumerate}
\item 
  $4\mid q-1$, $z=-1$, 
  $Y=A\rtimes_q \{-1\}$.
    \item $z^2\ne1$ and
    $Y=A\rtimes_q \{\Phi(z)\}$ or
$Y=(A\rtimes_q\{\Phi(z)\})\cup 
(A\rtimes_q\{\Phi(z)^{-1}\})
$ for some $\Phi\in \Gal(\ndF_{q_0}/\ndF_p)$.
\end{enumerate}
\end{enumerate}
\end{lem}

\begin{proof}
    Assume (1). Since $A\rtimes_qC$ is generated by $Y$, (1) and Lemma~\ref{lem:ACclasses} imply that $Y$ is a union of conjugacy classes $A\rtimes_q\{c\}$ with
    \[ \ord\,c=|C|=o(\cO). \]
    Therefore $\cO$ is split semisimple, and there exists $z\in \ndF_q^{\times 2}$ such that $\cO=\cO_{2,z}$.
    Moreover, $-1\in \ndF_q^{\times 2}$ and $Y=A\rtimes_q\{-1\}$ if $|C|=2$.
    Recall from the last part of Lemma~\ref{lem:realclasses} that any two distinct commuting non-involutive elements of $\cO$ are mutually inverse. Hence 
    (2) follows from Proposition~\ref{pro:ACintoPSL2}.
    
    Conversely, by applying the group homomorphisms appearing in Proposition~\ref{pro:ACintoPSL2}
    it follows that (2) implies (1).
\end{proof}

\begin{lem} \label{lem:PSLq0}
Let $q_0\in \ndN$ be such that $q_0\notin \{3,4\}$ and $q$ is a power of $q_0$. For any $(\PSL{2,q_0},Y)\in \cR$  the following are equivalent.
\begin{enumerate}
\item $(\PSL{2,q_0},Y)
\in \cRgen$ and $(\PSL{2,q_0},Y)\lR (\PSL{2,q},\cO)$.
\item One of the following conditions is fulfilled.
\begin{enumerate}
    \item $p=2$, $Y=\cO_{1,1}$, and $\cO=\cO_{1,1}$.
    \item $p>2$, $q_0>4$, $q$ is an even power of $q_0$, $\cO$ is unipotent, and $Y$ is a non-empty union of unipotent classes.
    \item $p>2$, $q_0>4$, $q$ is an odd power of $q_0$, $\cO$ is unipotent, and $Y$ is a unipotent class of $\PSL{2,q_0}$.
    \item $q_0>4$, $\cO$ is semisimple, and $Y$ is a semisimple conjugacy class of $\PSL{2,q_0}$ such that $\Phi([\chi_Y])=[\chi_\cO]$ for some $\Phi\in \Gal(\ndF_q/\ndF_p)$.
\end{enumerate}
\end{enumerate}
\end{lem}

Note that $\PSL{2,3}\cong \AG 4$ and $\PSL{2,4}\cong \AG 5$. Thus in view of Lemma~\ref{lem:A4classesinPSL} and Lemma~\ref{lem:A5classesinPSL} we may omit to discuss the cases with $q_0\in \{3,4\}$ in Lemma~\ref{lem:PSLq0}.

\begin{proof}
  We proceed with the proof by discussing different cases. In the first three cases we assume that $q_0=2$ or $\cO$ is unipotent, and deal with the remaining possibilities in the last case. 
  
  \emph{Case 1: $q_0=2$.} Then also $p=2$. If (1) holds, then $Y$ generates $\PSL{2,2}$, and hence $Y$ and $\cO$ consist of involutions. Thus (2)(a) holds. Conversely, (2)(a) implies (1) due to Lemma~\ref{lem:lR_oneclass}.

  \emph{Case 2: $q_0>4$, $p=2$, and $\cO$ is unipotent.}
  If (1) holds, then $Y$ consists of involutions, since $o(\cO)=p=2$; thus (2)(a) holds, since $\PSL{2,q_0}$ has only one class of involutions.
  As in case (1), (2)(a) implies (1).

  \emph{Case 3: $q_0>4$, $p>2$, and $\cO$ is unipotent.}
  Since $p$ is odd, besides $\cO$ there is another unipotent class $\cO'$ in $\PSL{2,q}$.
  If (1) holds and $q$ is an even power of $q_0$, then clearly (2)(b) holds.
  On the other hand, if (1) holds and $q$ is an odd power of $q_0$, then
  $Y$ consists of unipotent elements, and
  Proposition~\ref{pro:PSLintoPSL2} implies that each group embedding $f:\PSL{2,q_0}\to \PSL{2,q}$ sends distinct unipotent classes of $\PSL{2,q_0}$ to distinct unipotent classes of $\PSL{2,q}$. Indeed, these two classes are conjugate via $\lambda \id \in \PGL{2,q_0}$ for some $\lambda \in \ndF_{q_0}^\times \setminus \ndF_{q_0}^{\times 2}$, and now $\lambda \in \ndF_q^\times \setminus \ndF_q^{\times 2}$ since $q$ is an odd power of $q_0$.
  Therefore, (1) implies that $Y$ has to be a single unipotent class, which implies (2)(c).
  
  In the converse direction, both (2)(b) and (2)(c) imply the first part of (1) since $q_0>4$. Further, (2)(c) implies the second part of (1) due to Lemma~\ref{lem:lR_isoclasses}.
  Finally, assume that (2)(b) holds. Then $\PGL{2,q_0}$ is a subgroup of $\PSL{2,q}$, and hence both unipotent classes of $\PSL{2,q_0}$ can be embedded into a single unipotent class of $\PSL{2,q}$, and hence also into $\cO$ by the isomorphy of $\cO$ and the other unipotent class of $\PSL{2,q}$. Therefore again (1) holds.

  \emph{Case 4: $q_0>4$ and $\cO$ is semisimple.} We prove that (1) and (2)(d) are equivalent.
  
  Assume first that $Y$ generates $\PSL{2,q_0}$ and that
  \[ (\PSL{2,q_0},Y)\lR (\PSL{2,q},\cO).
  \]
  Then $Y$ is non-empty, and by Proposition~\ref{pro:PSLintoPSL2}
  there exist $g\in \PSL{2,q}$, $x\in \ndF_q^\times$ and $\Phi\in \Gal(\ndF_q/\ndF_p)$ such that
  \[
   \Phi(y)^{D_x^{-1}g} \in \cO
  \]
  for all $y\in Y$, where $D_x=\begin{bmatrix} x & 0 \\ 0 & 1 \end{bmatrix}$.
  Since characteristic polynomials are conjugation invariant, the latter implies that
  $\Phi([\chi_y])=[\chi_\cO]$ for all $y\in Y$. Consequently, $Y$ is a single semisimple conjugacy class according to Remark~\ref{rem:types_and_orders}(3) and the (assumed) semisimplicity of $\cO$.

  Conversely, assume that (2)(d) holds. Then $Y$ generates $\PSL{2,q_0}$ since $q_0>4$. Since $(\PSL{2,q},Y)\cong (\PSL{2,q},\Phi(y))$ via $\Phi$, we may assume that $\Phi=\id $. Since the characteristic polynomial of any $y\in Y$ remains invariant under the natural embedding $\PSL{2,q_0}\to \PSL{2,q}$, Remark~\ref{rem:types_and_orders} and the semisimplicity of $\cO$ imply (1).
  
  The above four cases cover all possibilities and hence the proof of the Lemma is complete.
\end{proof}

Recall that $\PGL{2,q}=\PSL{2,q}$ if $q$ is even. Moreover, if $q$ is odd, then the unipotent classes are contained in $\PSL{2,q}$, and hence they do not generate $\PGL{2,q}$.

\begin{lem} \label{lem:PGLq0}
Let  $q_0>3$ be odd such that $q$ is a power of $q_0^2$. For any injective rack $(\PGL{2,q_0},Y)\in \cR$ the following are equivalent.
\begin{enumerate}
\item $(\PGL{2,q_0},Y)\in \cRgen $ and $(\PGL{2,q_0},Y)\lR (\PSL{2,q},\cO)$.
 \item $\cO$ is semisimple,
 and one of the following conditions hold.
 \begin{enumerate}
  \item $o(\cO)=2$,
  and $Y$ is the conjugacy class of involutions of $\PGL{2,q_0}$ which is not contained in $\PSL{2,q_0}$, or the union of the two conjugacy classes of involutions
  of $\PGL{2,q_0}$.
  \item $o(\cO)>2$,
  $Y$ is a semisimple class of $\PGL{2,q_0}$, and there exists $\Phi\in \Gal(\ndF_q/\ndF_p)$ such that
    $\Phi([\chi_Y])=[\chi_{\cO}]$.
\end{enumerate}
\end{enumerate}
\end{lem}

Note that $\PGL{2,3}\cong \SG 4$. Thus in view of Lemma~\ref{lem:S4classesinPSL} we may omit to discuss the case $q_0=3$ in Lemma~\ref{lem:PGLq0}.

\begin{proof}
Note that the unipotent classes do not generate the group $\PGL{2,q_0}$. Therefore (1) implies that $\cO$ is semisimple. Moreover, $q_0>4$ by assumption. Therefore the argument in Case 4 of the proof of Lemma~\ref{lem:PSLq0} applies, with Proposition~\ref{pro:PSLintoPSL2} replaced by Remark~\ref{rem:PGLintoPSL2}.
Note however that $q$ is a power of $q_0^2$, and hence two monic quadratic polynomials over $\ndF_{q_0}$ are equivalent over $\ndF_q$ if and only if they are equivalent over $\ndF_{q_0}$ or they are equivalent to $x^2+d$ for some $d\in \ndF_q$. Therefore, if $\cO$ is the class of involutions of $\PSL{2,q}$, then a subrack $Y$ can consist of one or two classes of involutions of $\PGL{2,q_0}$ (exactly one of which is always in $\PSL{2,q_0}$).
\end{proof}

We now conclude from the previous claims the subrack structure of the conjugacy classes of $\PSL{2,q}$.

\begin{pro} \label{pro:unip_subracks}
    Let $x\in \ndF_q^\times$.
    The conjugacy class $\cO_{1,x}$ has the following subracks:
    \begin{enumerate}
        \item \label{it:abelian}
        abelian subracks of size at most $\frac{q-1}e$:
        conjugates of non-empty subsets of
        \[ \left\{
        \begin{bmatrix} 1 & bx \\
        0 & 1 \end{bmatrix}
        \,\Big|\,b\in \ndF_q^{\times 2} \right\}, \]
        \item \label{it:subPSL2} subracks isomorphic to a unipotent class of $\PSL{2,q_0}$, where $q$ is a power of $q_0$,
        \item \label{it:subPSL2_double} subracks isomorphic to the unipotent class $\cO_{1,1}$ of $\PGL{2,q_0}$, where $q$ is odd and a power of $q_0^2$,
        \item \label{it:dihedral_p=2}
        if $p=2$, then the class of involutions of the dihedral group $D_{2n}$, where $n$ divides $q-1$ or $q+1$,
        \item \label{it:alt5_p=2} if $p=2$ and $q$ is a power of $4$, then subracks isomorphic to the class of involutions in $\AG 5$,
        \item \label{it:alt4_p=3} if $p=3$, then subracks isomorphic to a conjugacy class of $3$-cycles in $\AG 4$,
        \item \label{it:alt4_p=3_double} if $p=3$ and $q$ is a power of $p^2$, then subracks isomorphic to the union of two conjugacy classes of $3$-cycles in $\AG 4$,
        \item \label{it:alt5_p=3} if $p=3$  and $q$ is a power of $p^2$, then subracks isomorphic to the class of $3$-cycles in $\AG 5$.
    \end{enumerate}
\end{pro}

\begin{proof}
  For each subrack $Y$ of $\cO_{1,x}$ let $H$ be the subgroup of $\PSL{2,q}$ generated by $Y$. In this case, $Y$ is a union of conjugacy classes of $H$. Conversely, unions of conjugacy classes in a subgroup, which generate the subgroup, are subracks of $\cO_{1,x}$. We go through all subgroups of $\PSL{2,q}$ using Dickson's theorem and list the possible subracks $Y$ identified before.
  
  The subracks related to elementary abelian $p$-subgroups have been discussed in Lemma~\ref{lem:elemabsubrack}  and appear in item~\eqref{it:abelian}.
  Subracks related to cyclic groups of order dividing $\frac{q-1}e$ or $\frac{q+1}e$ do not appear in $\cO_{1,x}$.
  Subracks related to the dihedral groups have been discussed in Lemma~\ref{lem:dihedralsubracks}. They appear here only if $p=2$. In this case they are covered by item~\eqref{it:dihedral_p=2}.
  Subracks related to the group $\AG 4$ have been determined in Lemma~\ref{lem:A4classesinPSL}. Since $\cO$ is unipotent, we obtain injective subracks only if $p=o(\cO)=3$.
  The solutions appear in items~\eqref{it:alt4_p=3} and \eqref{it:alt4_p=3_double}, see also Lemma~\ref{lem:realclasses}.
  Lemma~\ref{lem:S4classesinPSL} discusses subracks related to $\SG 4$. In our setting such subracks do not appear, since $o(\cO)\in \{2,4\}$ and $16\mid q^2-1$ are never satisfied for unipotent classes.

  Subracks of conjugacy classes of $\AG 5$ in case of $p\ne 5$ have been determined in Lemma~\ref{lem:A5classesinPSL}. They are contained in
  items~\eqref{it:alt5_p=2}
  and \eqref{it:alt5_p=3}.
  According to Lemma~\ref{lem:ACsubracksinPSL}, unipotent classes have no subracks arising from groups $A\rtimes_q C$.

  Next we turn to subracks coming from subgroups $\PSL{2,q_0}$ with $q_0\notin \{3,4\}$ and $q$ a power of $q_0$. They appear in Lemma~\ref{lem:PSLq0} in (2)(a), (2)(b), and (2)(c).
  The examples in (2)(a) cover case (2) of the Proposition for $p=2$ (except if $q_0=4$; the latter appeared already in case (5) of the Proposition).
  The examples in (2)(b) and (2)(c) with $Y$ a single class cover the remaining examples in case (2) of the Proposition (except the case $p=q_0=3$, which is itself covered by case (6) of the Proposition).
  The examples in Lemma~\ref{lem:PSLq0}(2)(b), which are unions of two unipotent classes of $\PSL{2,q_0}$ for $q_0>4$ odd, are also unipotent classes of $\PGL{2,q_0}$, since $q$ is an even power of $q_0$. These cover case (3) of the Proposition (except the case $q_0=3$, which is covered by case (7)).

  Finally, according to Lemma~\ref{lem:PGLq0}, the subgroups $\PGL{2,q_0}$ with $q_0>3$ odd do not give rise to subracks of unipotent classes which generate $\PGL{2,q_0}$.

  This completes the proof of the Proposition.
\end{proof}

A direct consequence of Proposition~\ref{pro:unip_subracks} is the following.

\begin{cor}  A unipotent class of $\PSL{2,q}$ is minimal non-abelian if and only if $q$ is prime.
\end{cor}

\begin{proof}
    If $q\ne p$, then each unipotent class of $\PSL{2,p}$ embeds as a subrack into each unipotent class of $\PSL{2,q}$. Unipotent classes of $\PSL{2,p}$ are always non-abelian according to Remark~\ref{rem:abelian_classes}.
    Conversely, if $q=p$, then the non-abelian subracks appearing in Proposition~\ref{pro:unip_subracks} coincide with the unipotent class of $\PSL{2,p}$ itself.
\end{proof}

Next we consider subracks of the class of involutions.

\begin{pro}
\label{pro:invol_subracks}
    Assume that $q$ is odd. The unique conjugacy class of involutions  is semisimple and contains the following subracks:
    \begin{enumerate}
        \item \label{it:inv_triv} subracks of size one,
        \item \label{it:inv_dih} subracks isomorphic to the set of all or, if $n$ is even, the union of two conjugacy classes of involutions of size $n/2$, in  dihedral groups of order $2n$, where $n\ge 2$ and either $n\mid \frac{q-1}2$ or $n\mid \frac{q+1}2$,
        \item \label{it:inv_S4}
        subracks isomorphic to the class $(1\,2)^{\SG 4}$ or the union $(1\,2)^{\SG 4}\cup (1\,2)(3\,4)^{\SG 4}$ of classes of $\SG 4$, if $16\mid q^2-1$,
        \item
        \label{it:inv_A5}
        subracks isomorphic to the class of involutions in $\AG 5$ if $5\mid q^2-1$,
        \item
        \label{it:inv_AC}
        if $4\mid q-1$, then subracks isomorphic to the conjugacy class $A\rtimes_q \{-1\}$ in $A\rtimes_q C$, where $|C|=2$ and $A$ is a non-zero subgroup of $\ndF_q$,
    \item \label{it:inv_PSL2q0}
        the class of involutions in $\PSL{2,q_0}$, where $q_0>4$ and $q$ is a power of $q_0$, and
    \item
    \label{it:inv_PGL2q0}
    the class of involutions in $\PGL{2,q_0}\setminus \PSL{2,q_0}$ or the disjoint union of the two classes of involutions in $\PGL{2,q_0}$,
    where $q_0>4$ and $q$ is a power of $q_0^2$.
    \end{enumerate}
\end{pro}

Recall that for each odd prime power $q_0$, the group $\PGL{2,q_0}$ has two classes of involutions, and exactly one of them is in $\PSL{2,q_0}$. Moreover, in item (7) we could also allow the case $q_0=3$, which however is already covered by item (3).
 
\begin{proof}
    We follow the strategy in the proof of Proposition~\ref{pro:unip_subracks}.

    Let $\cO$ be the unique class of involutions of $\PSL{2,q}$. This class is semisimple since $q$ is odd. Injective subracks $(G(Y),Y)$ of $\cO$ with $G(Y)$ an elementary abelian $p$-group do not appear, since $p\ne 2$.
    By Lemma~\ref{lem:subrackcyclic} the group $G(Y)$ is a cyclic group of order dividing $\frac{q-1}e$ or $\frac{q+1}e$ for some subrack $Y$ of $\cO$ if and only if $Y=\{y\}$ for some $y\in \cO$. This covers item~\eqref{it:inv_triv}.

    By Lemma~\ref{lem:dihedralsubracks}, the subracks $Y$ of $\cO$ with $G(Y)$ a dihedral group $D_{2n}$, where $n\ge 2$ and either $n\mid \frac{q-1}2$ or $n\mid \frac{q+1}2$, are as stated in item~\eqref{it:inv_dih}.

    According to Lemma~\ref{lem:A4classesinPSL}, the class of involutions has no subracks arising from the group $\AG 4$, since $o(\cO)\ne 3$.

    By Lemma~\ref{lem:S4classesinPSL}, the subracks arising from embeddings of unions of conjugacy classes of $\SG 4$ into $\cO$ are those listed in item~\eqref{it:inv_S4}.

    The subracks of $\cO$ arising from $\AG 5$ for $p\ne 5$ have been determined in Lemma~\ref{lem:A5classesinPSL}. They are covered by item~\eqref{it:inv_A5}.

The subracks of $\cO$ arising from the groups $A\rtimes_q C$ have been determined in Lemma~\ref{lem:ACsubracksinPSL}. They are covered by item~\eqref{it:inv_AC}.

Finally, the remaining subracks in \eqref{it:inv_PSL2q0} and \eqref{it:inv_PGL2q0} related to $\PSL{2,q_0}$ and $\PGL{2,q_0}$ have been determined in Lemmas~\ref{lem:PSLq0} and \ref{lem:PGLq0}.
\end{proof}

\begin{cor}
  Assume that $q$ is odd. The (semisimple) class of involutions of $\PSL{2,q}$ is
  \begin{enumerate}
      \item abelian if $q=3$, and
      \item neither abelian nor minimal non-abelian if $q>3$.
  \end{enumerate}
\end{cor}

\begin{proof}
    The class of involutions in $\PSL{2,3}\cong \AG 4$ is abelian, see also Remark~\ref{rem:abelian_classes}. If $q\ge 5$, then the class of involutions contains a non-abelian subrack of size $\frac{q+1}2>2$ coming from a dihedral subgroup of $\PSL{2,q}$. This subrack is strictly smaller than the class of involutions in $\PSL{2,q}$.
\end{proof}

Recall from Lemma~\ref{lem:nr_classesbyorder} that if $p\ne 3$ then $\PSL{2,q}$ has exactly one class of elements of order three.

\begin{pro}
\label{pro:o3_subracks}
  Assume that $p\ne 3$. The unique conjugacy class $\cO$ of order three elements of $\PSL{2,q}$ is semisimple and contains the following subracks:
    \begin{enumerate}
        \item \label{it:o3_triv} subracks of size one,
        \item 
        \label{it:o3_cyclic}
        the subracks $\{y,y^{-1}\}$, where $y\in \cO$,
        \item
        \label{it:o3_dih} if $q$ is odd or a power of 4, then subracks isomorphic to $(1\,2\,3)^{\AG 4}$ or to $(1\,2\,3)^{\AG 4}\cup (1\,3\,2)^{\AG 4}$,
        \item \label{it:o3_A5}
        if $5\mid q^2-1$, then subracks isomorphic to $(1\,2\,3)^{\AG 5}$,
        \item
        \label{it:o3_AC}
        if $3\mid q-1$ then
        subracks isomorphic to $A\rtimes_q \{z,z^{-1}\} \subseteq A\rtimes_q C_3$ or to
        $A\rtimes_q \{z\}$, where $z\in C_3\subseteq \ndF_q^{\times 2}$ has order 3, and
        \item \label{it:o3_PSL2q0}
        subracks isomorphic to the class of elements of order 3 in $\PSL{2,q_0}$, where $q_0>4$ and $q$ is a power of $q_0$.
    \end{enumerate}
\end{pro}

\begin{proof}
    Analogous to the proofs of Propositions~\ref{pro:unip_subracks} and \ref{pro:invol_subracks}.
    Note that elements of $\PGL{2,q_0}$ of order 3 have characteristic polynomial $x^2+x+1$ and are contained in $\PSL{2,q_0}$.
\end{proof}

\begin{cor}  Assume that $p\ne 3$. The unique conjugacy class of order three elements of $\PSL{2,q}$ is
\begin{enumerate}
    \item abelian if $q=2$, and
    \item minimal non-abelian if and only if $q=2^m$ for some odd prime $m$.
\end{enumerate}
If $q=2^m$ for some odd prime $m$, then all proper non-trivial subracks have size two.
\end{cor}

\begin{proof}
  The conjugacy class of 3-cycles in $\SG 3\cong \PSL{2,2}$ is abelian, see also Remark~\ref{rem:abelian_classes}. None of the other classes of order three elements are abelian.
  Let now $\cO$ denote the unique class of order three elements.
  If $q$ is odd or a power of 4, then $\cO$ contains $(1\,2\,3)^{\AG 4}$ as a proper non-abelian subrack of $\cO$.
  If $q=2^{mn}$ for some odd numbers $m,n\ge 3$, then $\cO$ contains the class of order three elements in $\PSL{2,2^m}$. Since $m>1$, this class is non-abelian.
  Finally, assume that $q=2^m$ for some odd prime $m$.
  Then $q=2^m$ is congruent to $-1$ mod 3, and $q^2=4^m$ is congruent to $-1$ mod 5. Therefore the 3-cycles in $\SG 3\cong \PSL{2,2}$ form the only non-trivial subrack of $\cO$ and hence $\cO$ is minimal non-abelian.
\end{proof}

We consider now the semisimple classes of elements of order at least four.

\begin{pro} \label{pro:semi_obe4}
   Let $\cO$ be a semisimple conjugacy class with $o(\cO)\ge 4$.
   Then $\cO$ has the following subracks.
   \begin{enumerate}
     \item \label{it:obe4_cyclic}
     the subracks $\{y\}$ and $\{y,y^{-1}\}$ with $y\in \cO$,
     \item
     \label{it:obe4_4cycles}
     if $16\mid q^2-1$ and $o(\cO)=4$, then subracks isomorphic to $(1\,2\,3\,4)^{\SG 4}$,
     \item
     \label{it:obe4_5cycles}
     if $5\mid q^2-1$ and $o(\cO)=5$, then subracks isomorphic to $(1\,2\,3\,4\,5)^{\AG 5}$,
     \item
     \label{it:obe4_ACracks}
     if $\cO$ is split semisimple, then subracks of the form
     $A\rtimes_q \{z\}\subseteq A\rtimes_qC$ or
     $A\rtimes_q \{z,z^{-1}\}\subseteq A\rtimes_q C$, where $z\in \ndF_q^{\times 2}$ with $\ord\,z=o(\cO)$, and $C$ is the cyclic group generated by $z$,
     \item
     \label{it:obe4_PGL2q0}
     if $[\chi_\cO]=x^2+x+d$ with $d\in \ndF_{q_0}^{\times}$ for a subfield $\ndF_{q_0}$ of $\ndF_q$ with $q_0>4$, then subracks isomorphic to the class $\cOd $ in $\PGL{2,q_0}$ with $[\chi_{\cOd}]=x^2+x+d$.
\end{enumerate}
\end{pro}

Note that in item~\eqref{it:obe4_PGL2q0} the number $d$ is in $\ndF_q^{\times 2}$, since $\cO\subseteq \PSL{2,q}$.
Moreover, if $d\in \ndF_{q_0}^{\times 2}$, then $\cOd\subseteq \PSL{2,q_0}$, and if $d\in \ndF_{q_0}\setminus \ndF_{q_0}^{\times 2}$, then $q$ is a power of $q_0^2$.

\begin{proof}
%[Proof of Proposition~\ref{pro:semi_obe4}]
    Analogous to the proofs of Propositions~\ref{pro:unip_subracks} and \ref{pro:invol_subracks}.
    Note that non-involutive elements of $\PGL{2,q}$ have a characteristic polynomial $x^2+x+d$ for a unique $d\in \ndF_q^{\times }$. Elements with this characteristic polynomial are in $\PSL{2,q}$ if and only if $d\in \ndF_q^{\times 2}$. Thus item~\eqref{it:obe4_PGL2q0} also covers subracks arising from subgroups $\PSL{2,q_0}$.
\end{proof}

\begin{cor}  A semisimple class $\cO$ of $\PSL{2,q}$ with $o(\cO)\ge 4$ is minimal non-abelian if and only if
\begin{enumerate}
    \item
$\cO$ is non-split semisimple, 
\item $[\chi_\cO]=x^2+x+d$ with $d\in \ndF_q^{\times 2}\setminus \ndF_{q_0}^\times$ for all proper subfields $\ndF_{q_0}$ of $\ndF_q$,
\item $16\nmid q^2-1$ or $o(\cO)\ne 4$, 
and
\item
$5\nmid q^2-1$ or $o(\cO)\ne 5$.
\end{enumerate}
Moreover, if $\cO$ is minimal non-abelian, then all non-trivial subracks have size two.
\end{cor}

\begin{proof}
This follows directly from Proposition~\ref{pro:semi_obe4}.
\end{proof}

\section{The associated groups of generating conjugacy classes}

In this section we are going to relate the associated group $\Ass X$ of a generating conjugacy class $X$ of a perfect group to the Schur multiplier of the group. As a corollary we determine the associated group of each non-trivial conjugacy class of $\PSL{2,q}$ with $q>3$.

Parts of the contents of this section have been known. We nonetheless include all proofs for completeness.

Let $(X,\triangleright) $ be a rack,
\[ F_X=\langle g_x\mid x\in X\rangle \]
be the free group on $X$, and
\[ R_0=\langle g_xg_yg_x^{-1}g_{x\triangleright y}^{-1} \mid x,y\in X \rangle \]
be the smallest normal subgroup of $F_X$ containing $g_xg_yg_x^{-1}g_{x\triangleright y}^{-1}$ for all $x,y\in X$.
The quotient
\[ \Ass X =F_X/R_0 \]
is known as the \emph{associated group} or as the \emph{enveloping group} or as the \emph{structure group} of the rack $X$. It has the following universal property. 

\begin{lem} \label{lem:univAss}
Let $(X,\triangleright )$ be a rack.
\begin{enumerate}
    \item The map $\iota_X:X\to \Ass X$, $x\mapsto g_x$, is a rack homomorphism, and $\iota_X(X)$ is stable under conjugation.
    \item
  Let $H$ be a group and $f:X\to H$ be a rack homomorphism. Then there is a unique group homomorphism $\overline{f}:\Ass X\to H$ such that the diagram
  % https://tikzcd.yichuanshen.de/#N4Igdg9gJgpgziAXAbVABwnAlgFyxMJZABgBpiBdUkANwEMAbAVxiRAA0QBfU9TXfIRQBGclVqMWbABLdeIDNjwEiZYePrNWiEAHEA+py7iYUAObwioAGYAnCAFskZEDghJRErW2shqDOgAjGAYABX5lIRAGGGscORt7J0QAJmo3D38gkPClQTZbLDMAC3jqTSkdAB0qiBoYWwYsMBhgay4EkDtHZ3T3VKzgsIj8nUKSsq9KkBr8HDpDbgouIA
  \[
\begin{tikzcd}
X \arrow[r, "f"] \arrow[d, "\iota_X"'] & H \\
\Ass X \arrow[ru, "\overline{f}"'] &  
\end{tikzcd}
\]
  commutes.
 In particular, if $f$ is injective, then $\iota_X$ is injective.
 \item For each injective rack homomorphism $f:X\to H$ to a group $H$, the group $\ker (\overline{f}:\Ass X\to H)$ is central in $\Ass X$.
\item Let $f$ and $\overline{f}$ be as in (2), and let $x\in X$. Assume that $f$ is injective. Then $\overline{f}^{-1}(C_H(f(x)))=C_{\Ass X}(\overline{x})=C_{\Ass X}\big(\overline{f}^{-1}(f(x))\big)$ for all $\overline{x}\in \Ass X$ with $\overline{f}(\overline{x})=f(x)$.
\end{enumerate}
\end{lem}

\begin{proof}
    (1) holds since
    \[ \iota_X(x\triangleright y)=g_{x\triangleright y}
    =g_xg_yg_x^{-1} \]
    for all $x,y\in X$ and since $\Ass X$ is generated by $\iota_X(X)$.

    (2) The unique group homomorphism $h:F_X\to H$ with $h(g_x)=f(x)$ factors through $R_0$ and yields the asserted unique group homomorphism $\overline{f}$.

    (3) Let $g\in \ker \overline{f}$. Then
    $\overline{f}(gg_xg^{-1})=\overline{f}(g_x)$ for all $x\in X$. Moreover, $gg_xg^{-1}=g_y$ for some $y\in X$ by the definition of $R_0$ and $\Ass X$, and since $\Ass X$ is generated by $\iota_X(X)$. By (2) and the injectivity of $f$ we conclude that $\iota_X$ is injective, and hence $y=x$. Therefore $gg_x=g_xg$ for all $x\in X$, and hence $g$ is central in $\Ass X$.
    
    (4) The inclusion $C_{\Ass X}(\overline{x})\subseteq \overline{f}^{-1}(C_H(f(x)))$ holds since $\overline{f}$ is a group homomorphism. Let now $\overline{x},g\in \Ass X$ such that
    \[ \overline{f}(\overline{x})=f(x),\quad \overline{f}(g)f(x)=f(x)\overline{f}(g). \]
    Then $g_x^{-1}\overline{x}\in \ker{\overline f}$, and by (3) there is a central $z\in \Ass X$ such that $\overline x=zg_x$.
    It follows that $C_{\Ass X}(\overline{f}^{-1}(f(x)))=C_{\Ass X}(\overline x)=C_{\Ass X}(g_x)$.
     Moreover,
\[ g\overline xg^{-1}=zgg_xg^{-1}=zg_x=\overline{x}. \]
Indeed, $gg_xg^{-1}=g_x$ since $gg_xg^{-1}=g_y$ for some $y\in X$, and the equations
\[ f(x)=\overline{f}(g_x)=\overline{f}(g_y)=f(y) \]
and the injectivity of $f$ imply that $x=y$. We conclude that
$g\in C_{\Ass X}(\overline x)$ and hence
$\overline{f}^{-1}(C_H(f(x)))=C_{\Ass X}(\overline{x})$.
\end{proof}

\begin{rem}
\label{rem:iotaXclass}
 Let $X$ be a subrack of a group $H$, and let $\psi:\Ass X\to H$ be the induced group homomorphism with $\psi(g_x)=x$ for all $x\in X$. Then \[ \psi(gg_xg^{-1})=\psi(g)x\psi(g)^{-1} \]
 for all $g\in \Ass X$ and $x\in X$. Since 
 $\iota_X(X)\subseteq \Ass X$ is conjugation invariant (by Lemma~\ref{lem:univAss}(1)), it follows that
 \begin{align} \label{eq:gxconjugates}
  g g_x g^{-1}=g_{\psi(g)x\psi(g)^{-1}}
\end{align}
for all $g\in \Ass X$, $x\in X$.
\end{rem}

In the remaining part of this section let $G$ be a group and $X$ be a conjugacy class of $G$. Assume moreover that $X$ generates $G$.
Let
\[ d:F_X\to \ndZ, \quad x\mapsto 1, \]
be the standard degree map, and let $F_X(0)=\ker d$. Then $F_X(0)$ is generated by the elements $g_xg_y^{-1}$ with $x,y\in X$. Moreover, let $\pi:F_X\to G$ be the group homomorphism with $\pi(g_x)=x$ for all $x\in X$, and let $R_X=\ker \pi$. Since $R_0\subseteq R_X$, $\pi $ induces a group homomorphism $\Ass X\to G$ which will also be denoted by $\pi$.

\begin{lem} \label{lem:R0}
The following hold.
    \begin{enumerate}
        \item $R_0\triangleleft F_X(0)$.
        \item
        $F_X(0)/R_0=[F_X/R_0,F_X/R_0]$ as subgroups of $F_X/R_0$.
        \item $[R_X,F_X]\subseteq R_0\subseteq R_X$.
    \end{enumerate}
\end{lem}

\begin{proof}
    (1) This is clear, since $R_0\subseteq \ker\,d$ and $R_0$ is normal in $F_X$.
    
    (2) The inclusion $\supseteq$ is obvious. We now prove the inclusion $\subseteq$.
    Let $x,y\in X$. On the one hand, in $Q=(F_X/R_0)/[F_X/R_0,F_X/R_0]$ the coset of  $g_xg_yg_x^{-1}g_{x\triangleright y}^{-1}$ is trivial, since 
    $g_xg_yg_x^{-1} g_{x\triangleright y}^{-1}\in R_0$. On the other hand, it coincides with the coset of $g_yg_{x\triangleright y}^{-1}$.
    Since $G$ is generated by $X$, it follows that  $g_x[F_X/R_0,F_X/R_0]=g_y[F_X/R_0,F_X/R_0]$ in $Q$ for any two conjugate elements $x,y\in X$. Since $X$ is a conjugacy class of $G$, it follows that
    $z[F_X/R_0,F_X/R_0]=0$ for all $z\in F_X(0)$, which proves the inclusion $\subseteq$.
    
    (3)
    The inclusion $R_0\subseteq R_X$ is obvious. By Remark~\ref{rem:iotaXclass} the equality
    \[gg_xg^{-1}= g_{\pi(g)x\pi(g)^{-1}}\]
    holds in $\Ass X$ for all $x\in X$ and $g\in \Ass X$. In particular,
    $mg_xm^{-1} =g_x$ in $\Ass X=F_X/R_0$ for all $m\in R_X$, $x\in X$. 
    Thus $[R_X,F_X]\subseteq R_0$.
\end{proof}

We continue with two claims on $\Ass X$,
the first of which is an application of Lemma~\ref{lem:R0}(3).

\begin{lem}
\label{lem:GactsonAss}
The adjoint action of $\Ass X$ gives rise to an action $\rightharpoonup$ of $G$ on $\Ass X$:
\[ \pi(g)\rightharpoonup h=ghg^{-1} \]
for all $g,h\in \Ass X$.
\end{lem}

\begin{proof}
It suffices to check that $ghg^{-1}=h$ for all $g\in \ker (\pi:\Ass X\to G)$ and all $h\in \Ass X$. Equivalently, it suffices to check that $[R_X,F_X]\subseteq R_0$. The latter holds by Lemma~\ref{lem:R0}(3).
\end{proof}

\begin{lem}
\label{lem:stem_GX}
Assume that $G$ is perfect. Then the short exact sequence
  \[ 1\to K\longrightarrow \Ass X\overset{\pi \times d}\longrightarrow G\times \ndZ \to 1 \]
 is a stem extension. It induces a central extension
  \[ 1\to K\longrightarrow D_X\overset{\pi |D_X}\longrightarrow G \to 1, \]
  where $D_X=[\Ass X,\Ass X]$ and $\pi|D_X$ is the restriction of $\pi$ to $D_X$.
\end{lem}

\begin{proof}
  Since $G$ is generated by $X$, the map $\pi:\Ass X\to G$ is surjective. Also, the map $d$ is surjective. Since $G$ is perfect, $\pi$ restricts to a surjective group homomorphism
  $D_X\to [G,G]=G$.

  Let $x_0\in X$. Recall that $\Ass X=D_X\langle g_{x_0}\rangle$.
  Then, by the previous paragraph, $\pi\times d$ is surjective. Let $K$ denote its kernel. 
  Note that $K=\ker \pi \cap D_X$.
  Moreover, $K$ commutes with $g_x$ for all $x\in X$ by Remark~\ref{rem:iotaXclass}, and hence $K$ is contained in the center of $\Ass X$. Hence the proof of the Lemma is completed.
\end{proof}

In \cite[I.3.8]{zbMATH03865539} the following universal coefficient theorem for central extensions of groups is given.
Let $Q$ be a group and $A$ be an abelian group. We write $Q_{ab}$ and $M(Q)$ for the abelianization $Q/[Q,Q]$ of $Q$ and the Schur multiplier of $Q$, respectively. Moreover, $\mathrm{Cext}(Q,A)$ is the group of congruence classes of \textit{central extensions} 
\[ 1\to A\to H\to Q\to 1\]
(that is, when $A$ is central in $H$)
and $\mathrm{Ext}(Q_{ab},A)$ the group of congruence classes of \textit{abelian extensions}
$1\to Q_{ab}\to H\to A$
(that is, extensions with abelian $H$).

\begin{thm} (Universal coefficient theorem, \cite[I.3.8]{zbMATH03865539})
Let $Q$ be a group and $A$ be an abelian group.
There is a natural short exact sequence
\[ 
1 \to \mathrm{Ext}(Q_{ab},A)\to \mathrm{Cext}(Q,A)
\to \mathrm{Hom}(M(Q),A)\to 1
\]
which is split.
\end{thm}

Recall from
\cite[Remark~5.4]{zbMATH03865539}
that for any short exact sequence
\[ 1\to R\to F\to Q\to 1, \]
where $Q$ is a perfect group and $F$ is a free group,
the factor group $[F,F]/[F,R]$ is a Schur covering group of $Q$ and $(R\cap [F,F])/[F,R]\cong M(Q)$.

\begin{lem} \label{lem:M(G)toK}
Assume that $G$ is perfect. There exist surjective group homomorphisms $\alpha: M(G)\to K$ and $\beta:[F,F]/[F,R]\to D_X$ such that
\[
\begin{tikzcd}
1 \arrow[r] & M(G) \arrow[r] \arrow[d, "\alpha"'] & {[F , F]} / [F,R] \arrow[r] \arrow[d, "\beta"'] & G \arrow[r] \arrow[d, double,no head] & 1 \\
1 \arrow[r] & K \arrow[r]                         & D_X \arrow[r,"\pi|D_X"]                         & G \arrow[r]                                & 1
\end{tikzcd}
\]
is a morphism of short exact sequences, where $\pi |D_X:D_X\to G$ is as in Lemma~\ref{lem:stem_GX}.
\end{lem}

\begin{proof}
According to 
\cite[I.3.2]{zbMATH03865539}, we may assume that $F=F_X$ and $R=R_X$.

The perfectness of $G$ implies that $G_{ab}$ is trivial. Hence the existence of $\alpha$ and $\beta$ follows from the universal coefficient theorem. In fact, in our setting the morphisms $\alpha$ and $\beta$ are induced by the identity on $F_X$. According to Lemma~\ref{lem:R0},
\[ D_X=F_X(0)/R_0=[F_X,F_X]/R_0 \]
and $[F_X,R_X]\subseteq R_0$. Thus $\beta$, and hence $\alpha$, are surjective.
\end{proof}

Let $x_0\in X$.

\begin{lem}
\label{lem:mux}
For all $x\in X$ there is a unique group homomorphism
\[ \mu_x: C_G(x)\to (R_X\cap [F_X,F_X])/[F_X,R_X] \]
such that
$\mu_x(g)=hyh^{-1}y^{-1}$ for all $g\in C_G(x)$ and $h\in F_X/[F_X,R_X]$ with $\pi(h)=g$, and for all $y\in \pi^{-1}(x)\subseteq F_X/[F_X,R_X]$.
Moreover,
\[ \mu_{f\triangleright x}(f\triangleright g)=\mu_x(g) \]
for all $f\in G$, $x\in X$, and $g\in C_G(x)$.
\end{lem}

\begin{proof}
  Let $y\in F_X/[F_X,R_X]$ and
  \[ \tilde{\mu}_y:
  F_X/[F_X,R_X]\to 
  [F_X,F_X]/[F_X,R_X], \quad 
  h\mapsto hyh^{-1}y^{-1}.
  \]
  Note that
  \begin{align} \label{eq:muinR}
  [\pi(h),\pi(y)]=1_G \Rightarrow \tilde{\mu}_y(h)\in (R_X\cap [F_X,F_X])/[F_X,R_X].
  \end{align}
  On the other hand,
  $\tilde{\mu}_y =\tilde{\mu}_{yr}$ for all $r\in R_X/[F_X,R_X]$, 
  since $R_X/[F_X,R_X]$ is contained in the center of $F_X/[F_X,R_X]$.
  Similarly,
  $\tilde{\mu}_y(h)=1$ for all $h\in R_X/[F_X,R_X]$ and
  \begin{align} \label{eq:tildemu} \tilde{\mu}_y(hh')=h\tilde{\mu}_y(h')h^{-1}\,\tilde{\mu}_y(h)
  \end{align}
  for all $h,h'\in F_X/[F_X,R_X]$. In particular, $\tilde{\mu}_y(hh')=\tilde{\mu}_y(h)$ for all $h\in F_X/[F_X,R_X]$, $h'\in R_X/[F_X,R_X]$.
  This proves that $\mu_x$ is well-defined. Moreover, \eqref{eq:muinR} and
  \eqref{eq:tildemu} also imply that $\mu_x$ is a group homomorphism.

  The last claim of the Lemma follows from the definition of $\tilde{\mu}_y$ and the centrality of $R_X/[F_X,R_X]$ in $F_X/[F_X,R_X]$.
\end{proof}

\begin{thm} \label{thm:relSchurmult}
  Let $G$ be a perfect group and $X$ be a conjugacy class of $G$. Assume that $G$ is generated by $X$.
  \begin{enumerate}
      \item The subgroup $\mu_x(C_G(x))$ of $M(G)$ is independent of the choice of $x\in X$. We call the factor group $M(G)/\mu_x(C_G(x))$ the \emph{Schur multiplier of $G$ relative to $X$} or the \emph{relative Schur multiplier} of the pair $(G,X)$.
      \item
      The group $D_X$ is a central extension of $G$ by $M(G)/\mu_x(C_G(x))$.
      \item Let $x_0\in X$, and let $g_0\in D_X$ be such that $\pi(g_0)=x_0$. The map
      \[ \Ass X \to D_X\times \ndZ,\quad
      g\mapsto (gg_{x_0}^{-d(g)}g_0^{d(g)},d(g)),
      \] is an isomorphism of groups.
  \end{enumerate}
  \end{thm}

\begin{proof}
  Claim (1) of the Theorem follows directly from the second half of Lemma~\ref{lem:mux}, since $G$ acts on $X$ transitively by conjugation.

  We prove (2).
  Let $H=G\times \ndZ$ and $f:X\to H$, $x\mapsto (x,1)$. Then $f$ is a rack morphism, and the map $\overline{f}:\Ass X\to H$ in Lemma~\ref{lem:univAss} is $\overline{f}=\pi \times d$. Let $K=\ker \overline{f}$ and $x\in X$.
  By Lemma~\ref{lem:univAss}(4),
  \begin{align*} C_{\Ass X}(g_xK)&=C_{\Ass X}(g_x)=\overline{f}^{-1}(C_H(f(x))\\
  &=\overline{f}^{-1}(C_H(x,1))
  =(\pi \times d)^{-1}(C_G(x)\times \ndZ)
  =\pi^{-1}(C_G(x)).
  \end{align*}
  Therefore, by the definition of $\mu_x$ in Lemma~\ref{lem:mux},  the kernel of the canonical map $\alpha:M(G)\to K$ in Lemma~\ref{lem:M(G)toK} contains $\mu_x(C_G(x))$. Conversely, 
let $\widehat{G}$ be the quotient of $F_X/[F_X,R_X]$ by the central subgroup $\mu_x(C_G(x))$. Then $\pi:F_X\to G$, $g_y\mapsto y$, induces a bijection $\widehat{\pi}$ between the conjugacy class of $g_x$ in $\widehat{G}$ and the class $X$ of $x$ in $G$. Then, by Lemma~\ref{lem:univAss} with the injective rack homomorphism $\widehat{\pi}^{-1}:X\to \widehat{G}$, there is a (necessarily surjective) group homomorphism $e:\Ass X\to \widehat{G}$ such that $e(g_y)=g_y$ for all $y\in X$. The induced morphisms
\[ D_X=\Ass X'\to \big([F_X,F_X]/[F_X,R_X]\big)/\mu_x(C_G(x))\]
and
\[ K\to \big((R_X\cap [F_X,F_X])/[F_X,R_X]\big)/\mu_x(C_G(x)) \]
are then also surjective, which implies that $K\cong M(G)/\mu_x(C_G(x))$.

(3) Note first that
$g_0g_{x_0}=g_{x_0}g_0$ by Equation~\eqref{eq:gxconjugates}, since $\pi(g_0)=x_0$. 

For all $g\in \Ass X$, $gg_{x_0}^{-d(g)}\in D_X$. Hence the map in the claim is well-defined. Moreover, that map is bijective with inverse
\[ D_X\times \ndZ\to \Ass X,\quad (g,n)\mapsto gg_0^{-n}g_{x_0}^n. \]
Finally, this last map is a group homomorphism. Indeed, for all $g,h\in D_X$ and $n,m\in \ndZ$ we obtain from Lemma~\ref{lem:GactsonAss} that
\[ (gg_0^{-n}g_{x_0}^n)(hg_0^{-m}g_{x_0}^m)=
g (\pi(g_0^{-n}g_{x_0}^n)\rightharpoonup h)g_0^{-n-m}g_{x_0}^{n+m}=ghg_0^{-n-m}g_{x_0}^{n+m}. \]
The proof of the Lemma is completed.
\end{proof}

We make the claim of Theorem~\ref{thm:relSchurmult} explicit in the case of the groups $\PSL{2,q}$ for $q>3$. Our proof is based on the knowledge of the Schur multipliers of the groups $\PSL{2,q}$, see e.g.~\cite{zbMATH04004443}; there is no proof there for the groups $\PSL{2,q}$, instead the author refers to the original papers  of Steinberg in \cite{zbMATH03427608} and in \cite{zbMATH05793093}.
The covering groups of the alternating groups have been known already to Schur
\cite{zbMATH02628889}.
For the statement we need the Schur covering group $\AG 6^*$, which we will discuss in more detail in Example~\ref{exa:A6cover}.

\begin{thm}
\label{thm:Schurcovering}
    \cite[Thm.\,7.1.1]{zbMATH04004443}
  Let $q>3$ be a prime power.
  \begin{enumerate}
      \item If $q\notin \{4,9\}$, then $\SL{2,q}$ is a Schur covering group of $\PSL{2,q}$, and $M(\PSL{2,q})$ is a cyclic group of order $\gcd(2,q-1)$.
      \item The group $\PSL{2,4}$ is isomorphic both to $\AG 5$ and to $\PSL{2,5}$. The group $\SL{2,5}$ is a Schur covering group of $\PSL{2,4}$, and $M(\PSL{2,4})$ is a cyclic group of order two.
      \item The group $\PSL{2,9}$ is isomorphic to $\AG 6$. The Schur multiplier of $\AG 6$ is cyclic of order 6. The group $\AG 6^*$ is a Schur covering group of $\AG 6$.
  \end{enumerate}
\end{thm}

In the following two claims we exclude the case $q=4$. However the group $\PSL{2,4}\cong \AG 5\cong \PSL{2,5}$ is still covered.

\begin{thm} \label{thm:relSchurPSL}
  Let $q>4$ be a prime power with $q\ne 9$, and let $X$ be a non-trivial conjugacy class of $\PSL{2,q}$.
  \begin{enumerate}
      \item 
      The Schur multiplier of $G$ relative to $X$ is
      \begin{enumerate}
          \item the trivial group if $q$ is even,
          \item the trivial group if $q$ is odd and $X$ is the class of involutions,
          \item the cyclic group of order two if $q$ is odd and $X$ is not the class of involutions.
      \end{enumerate}
      \item The associated group $\Ass X$ of $X$ is isomorphic to
      \begin{enumerate}
          \item the group $\SL{2,q}\times \ndZ$ if $q$ is even or $X$ is not the class of involutions,
          \item $\PSL{2,q}\times \ndZ$ if $q$ is odd and $X$ is the class of involutions.
      \end{enumerate}
  \end{enumerate}
\end{thm}

\begin{proof}
    Assume that $q$ is even. Then $M(\PSL{2,q})$ is the trivial group by Theorem~\ref{thm:Schurcovering}. Hence $\Ass X\cong \PSL{2,q}\times \ndZ$ (and $\PSL{2,q}=\SL{2,q}$) by Theorem~\ref{thm:relSchurmult}.

    Assume now that $q$ is odd and $X$ is not the class of involutions.
    Then $\SL{2,q}$ is a Schur covering group of $\PSL{2,q}$ by Theorem~\ref{thm:Schurcovering}, and $M(\PSL{2,q})$ is a cyclic group of order two. Let $y\in \SL{2,q}$ be a representative of an element $x\in X$. Since the trace of $y$ is non-zero, $-y$ is not conjugate to $y$. Hence $\mu_x(C_{\PSL{2,q}}(x))$ is the trivial group. Therefore $\SL{2,q}$ is a Schur multiplier of $\PSL{2,q}$ relative to $X$ by Theorem~\ref{thm:relSchurmult}, and $\Ass X \cong \SL{2,q}\times \ndZ$.

    Finally, assume that $q$ is odd and $X$ is the (unique and semisimple) class of involutions of $\PSL{2,q}$.
    Again, $\SL{2,q}$ is the Schur covering group of $\PSL{2,q}$.
    A representative of $x=\begin{bmatrix} 0 & 1 \\ -1 & 0 \end{bmatrix}\in X$ in $\SL{2,q}$ is the matrix
    $s=\begin{pmatrix} 0 & 1 \\ -1 & 0\end{pmatrix}$.
    The matrices $s$ and $-s$ are conjugate in $\SL{2,q}$ (more or less since there exist $a,b\in \ndF_q$ such that $a^2+b^2=-1$). Therefore $|\mu_x(C_{\PSL{2,q}}(x))|=2$ and the Schur multiplier of $\PSL{2,q}$ relative to $X$ is trivial.
    The remaining claim of the Theorem follows from Theorem~\ref{thm:relSchurmult}.
\end{proof}

As a Corollary we obtain explicit claims on the second quandle homology $H_2(X)$ of the non-trivial conjugacy classes $X$ of $\PSL{2,q}$ for $q>3$.

Recall from \cite[Th.\,1.12]{zbMATH06292119} that
\[ H_2(X)=C_{D_X}(g_x)_{ab}, \]
where $x$ is any element of $X$.

\begin{cor}
\label{cor:H2}
  Let $q>4$ be a prime power with $q\ne 9$, and let $X$ be a non-trivial conjugacy class of $\PSL{2,q}$.
  \begin{enumerate}
      \item
       If $q$ is even or $X$ is not the class of involutions,
  then $X$ embeds into $\SL{2,q}$, and $H_2(X)=C_{\SL{2,q}}(x)_{ab}$.
  \item If $q$ is odd and $X$ is the class of involutions, then for all $x\in X$,
  $C_{\PSL{2,q}}(x)$ is a dihedral group of order $q+1$ or $q-1$, and
  $H_2(X)=C_{\PSL{2,q}}(x)_{ab}$.
\end{enumerate}
\end{cor}

\begin{exa} \label{exa:A6cover}
  A very special example in the family of groups $\PSL{2,q}$ arises for $q=9$. The group $\PSL{2,9}$ is isomorphic to $\AG 6$. According to \cite[Thm.\,2.12.5]{zbMATH04004443}, which goes back to Schur, the group
  \begin{align*}
  \AG 6^ *=
  \langle g_1,g_2,g_3,g_4\mid &g_1^3z^3, g_2^2z^3, g_3^2z^3, g_4^2z^3, (g_1g_2z)^3,(g_2g_3z)^3, (g_3g_4z)^3,\\
  &(g_1g_3)^2z^3,(g_2g_4)^2z^3, (g_1g_4)^2z,z^6, [z,g_1], [z,g_2], [z,g_3], [z,g_4]\rangle \end{align*}
  is a Schur covering group of $\AG 6$, and its center has order 6 and is generated by $z$.
  Another (efficient) presentation of $\AG 6^*$ was given by Robertson in \cite{zbMATH03929276}:
  \[ \AG 6^*\cong
  \langle a,b\mid ab^3(ba)^{-4}, (ab^2ab^{-2})^2ab^2\rangle. \]
  There is a unique surjective group homomorphism
  \[ f:\AG 6^*\to \AG 6,
  z\mapsto (),\,
  g_1\mapsto (1\,2\,3),
  \,
  g_2\mapsto (1\,2)(3\,4),\,
  g_3\mapsto (1\,2)(4\,5),\,
  g_4\mapsto (1\,2)(5\,6).
  \]
  The conjugacy classes of $\AG 6$ are
  \[ ()^{\AG 6},
  (1\,2)(3\,4)^{\AG 6},
  (1\,2\,3)^{\AG 6},
  (1\,2\,3)(4\,5\,6)^{\AG 6},
  (1\,2\,3\,4)(5\,6)^{\AG 6},
  (1\,2\,3\,4\,5)^{\AG 6},
  (1\,2\,3\,4\,6)^{\AG 6}
  \]
  with sizes
  \[ 1,45,40,40,90,72,72. \]
  According to GAP, see \cite{GAP4}, the group $\AG 6^*$ has 31 conjugacy classes:
  \begin{center} 6 of size 1, 12 of size 72, 9 of size 90, 4 of size 120.
  \end{center}
  The classes of $\AG 6^*$ are coverings of the classes of $\AG 6$. Consequently,
  \begin{enumerate}
      \item the six classes of $\AG 6^*$ of size 1 are mapped by $f$ to $()^{\AG 6}$,
      \item
      the 12 classes of size 72 are mapped by $f$ to the classes $(1\,2\,3\,4\,5)^{\AG 6}$ and $(1\,2\,3\,4\,6)^ {\AG 6}$ of size 72,
      \item 6 classes of size 90 are mapped by $f$ to the class $(1\,2\,3\,4)(5\,6)^{\AG 6}$ of size 90,
      \item 3 classes of size 90 are mapped by $f$ to the class $(1\,2)(3\,4)^{\AG 6}$ of size 45, and
      \item the 4 classes of size 120 are mapped by $f$ to the two classes $(1\,2\,3)^{\AG 6}$ and $(1\,2\,3)(4\,5\,6)^ {\AG 6}$ of size 40. 
  \end{enumerate}
%  (According to GAP, the relation $z^6=1$ is superfluous.)
\end{exa}

\begin{thm} \label{thm:relSchurPSL29}
  Let $X$ be a non-trivial conjugacy class of $\PSL{2,9}$.
  \begin{enumerate}
      \item 
      The Schur multiplier of $G$ relative to $X$ is
      \begin{enumerate}
          \item the cyclic group of order two if $X=(1\,2\,3)^{\AG 6}$ or
          $X=(1\,2\,3)(4\,5\,6)^{\AG 6}$,       
          \item the cyclic group of order three if $X=(1\,2)(3\,4)^{\AG 6}$, and
          \item the cyclic group of order six if $X$ is one of the classes 
          $(1\,2\,3\,4)(5\,6)^{\AG 6}$, $(1\,2\,3\,4\,5)^{\AG 6}$,
          and $(1\,2\,3\,4\,6)^{\AG 6}$.
      \end{enumerate}
      \item The associated group $\Ass X$ of $X$ is isomorphic to
      the quotient of $\AG 6^*\times \ndZ$ by the unique central subgroup of $\AG 6^*$ of order $n(X)$, where
      \begin{align*} n( (1\,2\,3)^{\AG 6})&=
      n( (1\,2\,3)(4\,5\,6)^{\AG 6})=3,\\
      n( (1\,2)(3\,4)^{\AG 6})&=2, \\
      n((1\,2\,3\,4)(5\,6)^{\AG 6})&=n((1\,2\,3\,4\,5)^{\AG 6})=n((1\,2\,3\,4\,6)^{\AG 6})=1.
      \end{align*}
      \end{enumerate}
\end{thm}

\begin{proof}
  Apply Example~\ref{exa:A6cover} and Theorem~\ref{thm:relSchurmult}.    
\end{proof}

\bibliography{racks}
\bibliographystyle{alpha}

\end{document}